\documentclass[12pt]{amsart}
\usepackage{amsfonts}
\usepackage{amsmath,amssymb,amsbsy,amsfonts,amsthm,latexsym,
                        amsopn,amstext,amsxtra,euscript,amscd,mathrsfs,color,bm,cite}
\usepackage{float}
\usepackage[english]{babel}
\usepackage{mathtools}
\usepackage{url}
\usepackage[colorlinks,linkcolor=blue,anchorcolor=blue,citecolor=blue,backref=page]{hyperref}
\usepackage{cases}
\usepackage{txfonts}
\def\mB{\mathcal{B}}
\def\mN{\mathcal{N}}
\def\mR{\mathfrak{R}}
\def\mT{\mathcal{T}}
\def\mJ{\mathcal{J}}
\def\mV{\mathcal{V}}
\def\K{\mathrm{Kl}_2}
\def\r){\right)}
\def\l({\left(}
\def\B{\Bigg}
\def\d{\mathrm{d}}
\newcommand\bbm[1]{\mathbb{#1}}

\def\ssum{\mathop{\sum\nolimits^*}}

\def\ppmod{\!\!\!\!\!\pmod}
\def\ssqrt{\!\sqrt}

\let\ve=\varepsilon
\let\ol=\overline
\let\vp=\varphi

\theoremstyle{definition}

\newtheorem{remark}{Remark}[section]

\theoremstyle{plain}
\newtheorem{theorem}{Theorem}

\newtheorem{lemma}[theorem]{Lemma}

\numberwithin{equation}{section}
\numberwithin{theorem}{section}

\newcommand\ord{\mathop{\mathrm{ord}}\nolimits}

\begin{document}
\title{Bilinear forms with Kloosterman sums\\ and moments of twisted $L$-functions}

\author{Djordje Mili\'{c}evi\'{c}}
\address{D. Mili\'{c}evi\'{c}: Bryn Mawr College, Department of Mathematics, 101 North Merion Avenue, Bryn Mawr, PA 19010, USA}
\email{dmilicevic@brynmawr.edu}

\author[X. Qin] {Xinhua Qin}
\address{X. Qin: School of Mathematics, Hefei University of Technology, Hefei 230009, P.R. China}
\email{qinxh@mail.hfut.edu.cn}

\author[X. Wu]{Xiaosheng Wu}
\address{X. Wu: School of Mathematics, Hefei University of Technology, Hefei 230009, P.R. China}
\email{xswu@amss.ac.cn}

\thanks{D. Mili\'cevi\'c supported in part by the Simons Foundation Award MPS-TSM-00008085 and by the Charles Simonyi Endowment. X. Qin and X. Wu supported in part by the NSFC Grant 12271135, Anhui Provincial Natural Science Foundation Grant 2508085J005 and the Fundamental Research Funds for the Central Universities Grant JZ2025HGTG0254}


\begin{abstract}
We establish power-saving estimates for general bilinear forms with Kloosterman sums modulo arbitrary $q$, including when both variables are shorter than the P\'olya--Vinogradov range.
As an application, we obtain power-saving asymptotics for the second moment of (holomorphic or Maa\ss) modular $L$-functions twisted with Dirichlet characters to an arbitrary large admissible modulus $q$. The bounds obtained are independent of the Ramanujan--Petersson conjecture and remove all factorability conditions on $q$ in the work of Blomer, Fouvry, Kowalski, Michel, Mili\'cevi\'c, and Sawin.
\end{abstract}

\maketitle

\section{Introduction}

\subsection{Bilinear forms with Kloosterman sums}
Among the indispensable tools of analytic number theory are estimates on \emph{bilinear forms} of type
\begin{equation}
\label{bilinear-general-eq}
\sum_m\sum_n\alpha_m\beta_nK(mn),
\end{equation}
where $K:\mathbb{N}\to\mathbb{C}$ is a suitable arithmetic kernel depending on the particular situation and $\bm{\alpha}=(\alpha_m)$ and $\bm{\beta}=(\beta_n)$ are finitely supported sequences of coefficients on which available information is such that they need to be treated as essentially arbitrary beyond the size of $\|\bm{\alpha}\|_2=(\sum_m|\alpha_m|^2)^{1/2}$ and $\|\bm{\beta}\|_2=(\sum_n|\beta_n|^2)^{1/2}$. Of particular importance is the situation where $K:\mathbb{Z}/q\mathbb{Z}\to\mathbb{C}$ is a periodic kernel, often itself a ``trace function'' or otherwise a complete exponential or character sum modulo $q$, and coefficients are supported on intervals $m\leq M$, $n\leq N$ with $M$, $N$ as small as possible compared to the modulus $q$.
A major threshold for the lengths of the intervals is the P\'olya--Vinogradov range $M,N\sim q^{\frac12+\ve}$; for $M$ and $N$ in shorter ranges, the completion method through the Fourier transform does not make the sums shorter.

Following approaches such as those we discuss in \S\ref{sketch-sec}, at the heart of the proofs of such estimates are bounds on complete \emph{sums of products} (such as our \eqref{eqdefS}), which are highly sensitive to the arithmetic structure of the modulus $q$. We refer the reader to \cite{KMS2017,FKM2015} for compherensive surveys and several foundational results.

When studying automorphic forms and $L$-functions, the kernel $K(n)=\K(n;q)$  given by the normalized Kloosterman sum
\[
\K(a;q)=\frac1{\sqrt{q}}\ssum_{x\ppmod{q}}e_q(ax+\ol{x})
\]
often arises naturally through applications of the classical Kuznetsov, Petersson, and Voronoi formulas.
Power-saving bounds breaking the P\'olya--Vinogradov threshold on the bilinear form \eqref{bilinear-general-eq} with these Kloosterman sums as the kernel played a critical role in the proofs of power-saving asymptotics for the second moment of modular $L$-functions in the series of papers by Blomer, Fouvry, Kowalski, Michel, Mili\'cevi\'c, and Sawin~\cite{BM2015,BFK+17a,KMS2017,BFKMMS}, which treated \eqref{bilinear-general-eq} in critical ranges by disjoint methods, with \cite{BM2015} covering the case of suitably factorable $q$ by Weyl differencing and $p$-adic methods and \cite{KMS2017} covering prime $q$ by deep algebro-geometric methods including the machinery of trace functions. From the perspective of the arithmetic structure of the modulus $q$, these can be thought of as extreme cases of prime and well-factorable (including possibly powerful) $q$, leaving unaddressed a number of intermediate regimes such as $q=q_1q_2$ and $q=q_1^2$ (with $q_j$ prime and $q_1\asymp q_2$) that do not neatly follow either of the two paradigms. For applications, one is naturally interested in results that apply to all moduli $q$ uniformly, and our first result establishes such a bound on \eqref{bilinear-general-eq}.

\begin{theorem}\label{thmmain}
Let $q$ be a positive integer, $M,N\ge1$, and let $\bm{\alpha}=(\alpha_m)$, $\bm{\beta}=(\beta_n)$ be two sequences supported respectively on $[1,M]$ and $[1,N]$. If the conditions
\begin{align}\label{conditionMN}
1\le M\le Nq^{\frac14}, \quad M^{\frac{7}{5}}N<q^{\frac32}, \quad MN\le q^{\frac 54}
\end{align}
are satisfied, then for any integer $c$ coprime with $q$,
we have
\begin{equation}
\label{eqmain}
\begin{aligned}
&\sum_{m\le M}\sum_{n\le N} \alpha_m \beta_n \K(cmn;q)\\
&\qquad\ll q^\ve\|\bm{\alpha}\|_2\|\bm{\beta}\|_2(MN)^{\frac12}
\l(M^{-\frac12}q^{\frac1{6}}+M^{-\frac3{25}}N^{-\frac3{10}}q^{\frac15}+(MN)^{-\frac3{16}}q^{\frac{11}{64}}\r).
\end{aligned}
\end{equation}
\end{theorem}

\begin{remark}
The principal novelty in Theorem~\ref{thmmain} over \cite{BM2015,KMS2017} is that it holds uniformly over all moduli $q$. As a baseline, applying Weil's bound directly gives the trivial estimate $\|\bm{\alpha}\|_2\|\bm{\beta}\|_2(MN)^{\frac12}$ for the bilinear form. A bit more refined analysis, following the P\'olya--Vinogradov method, yields the improved bound (see \cite[Theorem 1.17]{FKM14})
\begin{equation}
\label{eqPoly}
\begin{aligned}
\sum_{m\le M}\sum_{n\le N} \alpha_m \beta_n \K(cmn;q)\ll \|\bm{\alpha}\|_2\|\bm{\beta}\|_2(MN)^{\frac12}
\l(q^{-\frac1{4}}+M^{-\frac1{2}}+N^{-\frac1{2}}q^{\frac{1}{4}} \log q\r),
\end{aligned}
\end{equation}
which requires at least $N\ggg q^{\frac12}\log^2q$ to improve upon the trivial bound. To contrast, in the balanced range $M\sim N$, which is the most important one in many applications including our Theorem~\ref{thmL}, the bound \eqref{eqmain} is nontrivial already for $M\sim N\gg q^{\frac{10}{21}+\delta}$  (and saves $q^{-\frac1{100}+\ve}$ in the P\'olya--Vinogradov range $M,N\sim q^{1/2+\ve}$).

We point out that Pascadi~\cite{Pascadi2025} has simultaneously and independently obtained power-saving bounds on bilinear forms in Kloosterman sums and asymptotics for moments of twisted $L$-functions similar to our Theorems \ref{thmmain} and \ref{thmL}. The two papers are complementary, and the two substantially different methods perform slightly better in different ranges and for different types of moduli, both pleasingly achieving power savings for bilinear sums of square-root length (and substantially beyond) and general moduli. The methods of the present paper (which uses algebraic geometry and builds on Kowalski--Michel--Sawin~\cite{KMS2017} and Blomer--Mili\'cevi\'c~\cite{BM2015}) perform better for general moduli $q$ and remove the dependence on the Ramanujan--Petersson conjecture in Theorem \ref{thmL}, while the methods of \cite{Pascadi2025} (which uses non-abelian Fourier analysis and builds on Shkredov \cite {Shkredov2018,Shkredov2021}) perform better and in longer ranges for specific classes of moduli and can handle more general ranges of variables.
\end{remark}

\subsection{Moments of twisted $L$-functions}
Asymptotics for moments with a power-saving error term are crucial to the amplification and related analytic techniques in questions such as subconvexity, nonvanishing, and extreme values of $L$-functions; see \cite{BFKMMS} for a rich sample of applications such as these and their implications for distribution of analytic ranks, modular symbols, and more.

For $f$ a holomorphic or Maa{\ss} cuspidal newform of level 1, a natural family arises when considering the central values $L(1/2,f\otimes\chi)$
twisted by all primitive characters to a large modulus $q$, which may be seen as the finite place analogue of the family of archimedean twists $L(1/2+it,f)$ along the critical line. Following a breakthrough power-saving asymptotic of Young \cite{You11} for the fourth moment $L(1/2,\chi)^4$ (the Eisenstein series analogue), power-saving asymptotics for the twisted second moment (such as \eqref{eqmoment} below) for $f_j$ holomorphic cusp forms (or Maa{\ss} cusp forms satisfying the Ramanujan--Petersson conjecture) were obtained via bounds on bilinear forms with Kloosterman sums such as our Theorem~\ref{thmmain} in the cases of suitably factorable $q$~\cite{BM2015} and prime $q$~\cite{KMS2017}. In addition to these significant factorability conditions, the applicability of these results was also hampered by the dependence on the Ramanujan--Petersson conjecture as well as the technical condition that $q$ not be highly divisible by prime 2.

As an important application of Theorem~\ref{thmmain} we obtain the following evaluation of the twisted second moment, which is free of any dependence on the Ramanujan--Petersson conjecture and any factorability conditions on the modulus $q$. Let $\vp^*(q)$ denote the number of primitive Dirichlet characters modulo $q$; such characters exist iff $q\not\equiv 2\pmod 4$, in which case we term $q$ admissible and we have $\vp^*(q)=q^{1+o(1)}$.
As in \cite{BM2015}, the leading constants in the asymptotic evaluation are expressed in terms of the finite Euler products
\begin{equation}
\label{euler-products-eq}
\begin{gathered}
P(s)=\prod_{p\mid q}\l(1-\frac{\lambda_1(p^2)}{p^s}+\frac{\lambda_1(p^2)}{p^{2s}}-\frac{1}{p^{3s}}\r)\l(1-\frac{1}{p^{2s}}\r)^{-1},\\
Q(s)=\prod_{p\mid q}\l(1-\frac{\lambda_1(p)\lambda_2(p)}{p^s}+\frac{\lambda_1(p^2)+\lambda_2(p^2)}{p^{2s}}-\frac{\lambda_1(p)\lambda_2(p)}{p^{3s}}+\frac1{p^{4s}}\r) \l(1-\frac{1}{p^{2s}}\r)^{-1}.
\end{gathered}
\end{equation}

Our evaluation of the twisted second moment over arbitrary admissible moduli $q$ is then as follows.

\begin{theorem}\label{thmL}
For $j=1,2$, let $f_j$ be both holomorphic or both Maa{\ss} cuspidal newforms of level $1$ with Hecke eigenvalues $\lambda_j(n)$.
If their root numbers satisfy $\ve(f_1)\ve(f_2)=1$, then
\begin{align}\label{eqmoment}
\frac1{\vp^*(q)}\ssum_{\chi \ppmod q}L(1/2,f_1\otimes\chi)\ol{L(1/2,f_2\otimes\chi)}=\frac2{\zeta(2)}M(f_1,f_2,q)+O\l(q^{-\frac1{216}+\ve}\r),
\end{align}
where
\begin{align}
M(f_1,f_2,q)=\begin{cases}
P(1)L(1,{\rm{sym}^2}f_1)\l(\log q+c+\frac{P'(1)}{P(1)}\r), &f_1=f_2,\\
Q(1)L(1,f_1\times f_2), & f_1\neq f_2,
\end{cases}
\end{align}
where $P(s)$, $Q(s)$ are finite Euler products shown in \eqref{euler-products-eq}, and $c$ is a constant depending only on $f_1$ (not on $q$).
\end{theorem}

The main terms in Theorem~\ref{thmL} feature leading coefficients $L(1,\mathrm{sym}^2f_1)$ and $L(1,f_1\times f_2)$, which do not vanish by the lower bounds of Hoffstein--Lockhart~\cite{HL94}, Ramakrishnan--Wang~\cite{RW03}, and Brumley~\cite{Bru06}. In light of $P(s)=(L_q(s,\mathop{\mathrm{sym}}^2f_1)/\zeta_q(2s))^{-1}$ and $Q(s)=(L_q(s,f_1\times f_2)/\zeta_q(2s))^{-1}$, we have $P(1),Q(1)=\exp(O_{\ve}((\log\log q)^{\ve}))$
and $P'(1)/P(1)=O_{\ve}((\log\log q)^{1+\ve})$, so that the main terms in Theorem~\ref{thmL} are not far from a linear polynomial in $\log q$ or a constant depending on $f_1$ and $f_2$; see Remark~\ref{PQ-remark}.

\subsection{Notation}
We follow the standard convention where $\ve$ denotes an arbitrarily small positive constant that may vary at each occurrence. As is common, we write $e(z)=e^{2\pi i z}$; we denote $f=O(g)$ or $f\ll g$ to indicate that $|f|\leqslant Cg$ for some constant $C>0$ which is allowed to depend on $\ve$ (with $\ve$ allowed to vary from line to line as explained above)
and (in bounds pertaining to Theorem~\ref{thmL} and its proof) on the fixed cusp forms $f_1$ and $f_2$,
 but is otherwise absolute/independent from all other parameters unless specifically indicated by a subscript, and we write $f\asymp g$ to mean that $f\ll g$ and $g\ll f$ both hold.
For $d, q\in \mathbb{N}$, we denote by $q_d$ the maximal factor of $q$ such that $(q_d,d)=1$.

\section{Sketching the treatment for bilinear forms}

Theorem \ref{thmmain} is proved by combining (and then optimizing) the results of Theorem \ref{thmBKs} and Theorem \ref{thmBK}, which provide power-saving estimates on the bilinear forms with Kloosterman sums according to the factorization of the modulus $q$. In \S\ref{statements-subsec}, we state these two theorems; then, in \S\ref{thmmain-subsec}, we show how they combine to prove Theorem~\ref{thmmain}, and in \S\ref{sketch-sec}, we give a general overview of their proofs, which then constitute most of the rest of the paper (sections~\ref{secwfm}--\ref{variety-sec}).

\subsection{Estimates on bilinear forms in the factorable and rough regimes}
\label{statements-subsec}
In this section, we state Theorems~\ref{thmBKs} and \ref{thmBK}.

When $q$ admits a favorable factorization such that there exists a divisor $s\mid q$ satisfying $s\in [q^{\delta_1}, q^{\frac12-\delta_2}]$ with positive constants $\delta_1,\delta_2$ not excessively small, we establish the estimate for the bilinear form via the following theorem.
\begin{theorem}\label{thmBKs}
Let $M, N\ge1$ and $s, q\in\mathbb{N}$ with $s\mid q$. Given two sequences $\bm{\alpha}=(\alpha_m)$ and $\bm{\beta}=(\beta_n)$ supported respectively on $[M, 2M]$ and $[N, 2N]$, then for any integer $c$ coprime with $q$, it holds that
\begin{equation}
\label{eqBKs}
\begin{aligned}
&\sum_{M\le m\le 2M}\sum_{N\le n\le 2N} \alpha_m \beta_n \K(cmn;q)\\
&\qquad\ll q^\ve\|\bm{\alpha}\|_2\|\bm{\beta}\|_2(MN)^{\frac12}
\l(M^{-\frac12}s^{\frac12}+q^{-\frac14}s^{\frac14}+N^{-\frac12}q^{\frac14}s^{-\frac14}\r).
\end{aligned}
\end{equation}
\end{theorem}

When $q$ lacks favorable factorization (i.e., is essentially a prime or a product of two, not necessarily distinct, similarly-sized primes), we adopt the decomposition:
\begin{align}\label{defqrho}
q= q^\star\rho,
\end{align}
where $q^\star$ denotes either a prime number or a product of two primes. Let $p_{\min}$ represent the smallest prime factor of $q^{\star}$. This scenario is addressed by the subsequent theorem.

\begin{theorem} \label{thmBK}
Let $q=q^{\star}\rho$ as defined in \eqref{defqrho} with $p_{\min}\ge(q/\rho)^{\frac13}$. For real numbers $M$ and $N$ satisfying
\begin{align}\label{eqcBK}
1\le M\le Nq^{\frac14},\quad MN\le q^{\frac54} \rho^{-\frac14},
\end{align}
let $\mN$ be an integer interval of length $\lfloor N\rfloor$, and let $\bm{\alpha}=(\alpha_m)$ and $\bm{\beta}=(\beta_n)$ be complex sequences supported respectively on $[M, 2M]$ and $\mN$. Then for any integer $c$ coprime with $q$, we have
\begin{equation}
\label{eqBKrho}
\begin{aligned}
&\sum_{M\le m\le 2M}\sum_{n\in\mN} \alpha_m \beta_n \K(cmn;q)\\
&\qquad\ll q^\ve\|\bm{\alpha}\|_2\|\bm{\beta}\|_2(MN)^{\frac12}\l(M^{-\frac12}+p_{\min}^{-\frac12}+(MN)^{-\frac3{16}}q^{\frac{11}{64}}\rho^{\frac{9}{64}}\r).
\end{aligned}
\end{equation}
\end{theorem}

\subsection{Proof of Theorem \ref{thmmain}}
\label{thmmain-subsec}
In this section, we prove Theorem~\ref{thmmain} by combining the results of Theorem~\ref{thmBKs} and Theorem~\ref{thmBK}.

\begin{proof}[Proof of Theorem~\ref{thmmain}]
The estimate \eqref{eqmain} becomes nontrivial under the condition
\[
 M^{\frac25}N\gg q^{\frac23},
\]
a convention we maintain throughout the subsequent proof. For given $q$, we take
\[
\rho=\prod_{ p\le q^{1/3},\ p^k\parallel q}p^k.
\]
When $\rho\neq q$, the component $q^\star$ exhibits one of two forms: either a single prime or a product of two primes, with the smallest prime factor $p_{\min}>q^{\frac1{3}}$.
The critical threshold
\[
\rho_0=M^{\frac{12}{25}}N^{-\frac 45}q^{\frac15}
\]
emerges as the equilibrium point where the last terms in \eqref{eqBKs} and \eqref{eqBKrho} concide when $s=\rho=\rho_0$.
Through systematic analysis of the constrains $MN^{-1}\ll q^{\frac14}$, $M^{\frac25}N\gg q^{\frac{2}{3}}$, and $MN\le q^{\frac54}$, we derive two refined bounds for $\rho_0$ such that
\begin{align}\label{eqrho01}
\rho_0&=(MN^{-1})^{\frac{4}{7}}(M^{\frac25}N)^{-\frac{8}{35}}q^{\frac15}\ll (MN^{-1})^{\frac25}q^{\frac{19}{210}},\\
\label{eqrho02}
\rho_0&=N^{-\frac23}(MN^{-1})^{\frac{23}{75}}(MN)^{\frac{13}{75}}q^{\frac15}\ll N^{-\frac23}q^{\frac{37}{75}}.
\end{align}
The proof strategy bifurcates based on the relative magnitude of $\rho$ compared to $\max\{1,\rho_0\}$:

\emph{Case I:} $\rho\le \max\{1, \rho_0\}$.
Condition \eqref{eqcBK} is ensured by \eqref{conditionMN}.
Applying Theorem \ref{thmBK} yields
\begin{align*}
&\sum_{m\le M}\sum_{n\in\mN} \alpha_m \beta_n \K(cmn;q)\ll q^\ve\|\bm{\alpha}\|_2\|\bm{\beta}\|_2(MN)^{\frac12}\l(M^{-\frac12}+q^{-\frac1{6}}+(MN)^{-\frac3{16}}q^{\frac{11}{64}}\rho^{\frac{9}{64}}\r)\\
&\qquad\qquad\ll q^\ve\|\bm{\alpha}\|_2\|\bm{\beta}\|_2(MN)^{\frac12}\l(M^{-\frac12}+q^{-\frac1{6}}+M^{-\frac3{25}}N^{-\frac3{10}}q^{\frac15}+(MN)^{-\frac3{16}}q^{\frac{11}{64}}\r).
\end{align*}

\emph{Case II:} $\rho>\max\{1, \rho_0\}$.
Here, there exists an integer $s\mid\rho$ satisfying
\begin{align}\label{eqconditions}
\rho_0\le s\le \max\{q^{\frac13}, \rho_0^2\}.
\end{align}
To justify this, observe that if any prime $p\mid \rho$ satisfies $\rho_0\le p \le q^{\frac1{3}}$, we set $s=p$. Otherwise, all prime factors of $\rho$ are smaller than $\rho_0$, and the existence of such $s$ is obvious. Applying Theorem \ref{thmBKs} with this $s$ gives
\begin{align*}
&\sum_{m\le M}\sum_{n\in\mN} \alpha_m \beta_n \K(cmn;q)\\
&\qquad\qquad\ll q^\ve\|\bm{\alpha}\|_2\|\bm{\beta}\|_2(MN)^{\frac12}
\l(M^{-\frac12}\l(q^{\frac16}+\rho_0\r)+q^{-\frac{1}{4}}\l(q^{\frac1{12}}+\rho_0^{\frac12}\r)+N^{-\frac12}q^{\frac14}\rho_0^{-\frac14}\r).
\end{align*}
From \eqref{eqrho01} and \eqref{eqrho02}, it follows that
\[
M^{-\frac12}\rho_0\ll N^{-\frac12}q^{\frac14}\rho_0^{-\frac14},\quad q^{-\frac14}\rho_0^{\frac12}\ll N^{-\frac12}q^{\frac14}\rho_0^{-\frac14},
\]
which simplifies the bound to
\begin{align*}
\sum_{m\le M}\sum_{n\in\mN} \alpha_m \beta_n \K(cmn;q)&\ll q^\ve\|\bm{\alpha}\|_2\|\bm{\beta}\|_2(MN)^{\frac12}
\l(M^{-\frac12}q^{\frac16}+q^{-\frac{1}{6}}+M^{-\frac3{25}}N^{-\frac3{10}}q^{\frac15}\r).
\end{align*}
Combining the estimates from both cases completes the proof of Theorem \ref{thmmain}.
\end{proof}

\subsection{Sketch of the treatments for Theorems \ref{thmBKs} \& \ref{thmBK}}
\label{sketch-sec}
An application of the Cauchy--Schwarz inequality gives
\begin{align}\label{eqthm5BM}
\bigg|\sum_{m\le m\le 2M}\sum_{n\in\mN} \alpha_m \beta_n \K(cmn;q)\bigg|^2\ll \|\bm{\beta}\|_2^2\sum_{n\in\mN}\B|\sum_{M\le m\le 2M}\alpha_m\K(cmn;q)\B|^2,
\end{align}
which reduces the problem to estimating the sum
\begin{align}\label{eqSumKlo}
\sum_{n\in\mN}\B|\sum_{M\le m\le 2M}\alpha_m\K(cmn;q)\B|^2.
\end{align}
The analysis of \eqref{eqSumKlo} bifurcates according to the factorization structure of $q$, requiring fundamentally distinct methods.
When $q$ admits a favorable factorization, we employ ``Weyl differencing'' to reduce the modulus prior the deeper Kloosterman sums analysis. This approach is detailed in \S \ref{secwfm}, adapting the methodology of \cite[Section 9]{BM2015} while eliminating auxiliary conditions required in \cite[Theorem 5]{BM2015}. When $q$ has large prime factors, we effectively reduce the modulus to the product of the two largest prime factors. Through ``shifted by $ab$'' trick, we transfer the problem to handling multiplicative completed exponential sums. For prime moduli, these were rigorously treated by Kowalski, Michel, and Sawin \cite{KMS2017}. The most intricate case occurs for the product of two primes, where we establish bounds by enumerating rational points on a nearly zero-dimensional variety in $\bbm{F}_p^{11}$.

\section{Bilinear forms for well-factorable moduli}\label{secwfm}
In this section, we give the proof of Theorem \ref{thmBKs} where $q$ admits a favorable factorization.
Given a divisor $s$ of $q$ with suitable size, we first apply ``Weyl differencing'' to reduce the modulus, then obtain power savings through careful analysis of Kloosterman sums, where the saving depends critically on the parameter $s$.
This follows the approach of the work of Blomer and Mili\'cevi\'c \cite[Theorem 5]{BM2015}, who established the bound
\begin{align}\label{eqthm5BM2015}
\sum_{\substack{M\le m\le 2M\\ (m,q)=1}}\B|\sum_{K\le k\le 2K}\lambda(k)\K(ckm;q)\B|^2\ll (KMq)^\ve\B(KMs+\frac{K^2Ms^{\frac12}}{q^{\frac12}}+\frac{K^2q^{\frac12}}{s^{\frac12}}\B)
\end{align}
under the restrictive assumptions that $\lambda(k)\le 1$, $(q/s,2)=1$, and $(m,q)=1$. These conditions prevent direct application to our setting.

The following theorem presents a refined version of \eqref{eqthm5BM2015}, where we eliminate all technical constraints.
For consistency with the literature, we maintain the original notation from \cite[Theorem 5]{BM2015}, while adopting the standard Kloosterman sums.

\begin{theorem}\label{lemSumMK}
Let $M,K\ge1$, $s,q\in\mathbb{N}$ with
$s\mid q$.
Let $\lambda: \mathbb{N}\rightarrow\mathbb{C}$ be an arithmetic function supported on $[K, 2K]$. For any integer $c$ coprime with $q$, we have
\begin{align}\label{eqSumMK}
\sum_{M\le m\le 2M}\B|\sum_{K\le k\le 2K}\lambda(k)\K(ckm;q)\B|^2\ll (qKM)^\ve\|\lambda\|_2\B(Ms+\frac{KMs^{\frac12}}{q^{\frac12}}+\frac{Kq^{\frac12}}{s^{\frac12}}\B).
\end{align}
\end{theorem}
The removal of the coprimality condition $(m,q)=1$ obviates the need to introduce a factor $r\mid q$ as in \cite[Theorem 5]{BM2015}, since it makes no essential difference in the context of M\"obius inversion. We therefore set $r=q$ without loss of generality. Theorem \ref{thmBKs} follows immediately upon substituting the estimate \eqref{eqSumMK} into \eqref{eqthm5BM}.

\subsection{The case for high powers of $2$}
The condition $(q/s,2)=1$ affects the bound \eqref{eqthm5BM2015} only when $q$ is divisible by extremely high powers of $2$. To fill this gap, we establish an estimate as that in \cite[Propostion 23]{BM2015} for high powers of $2$.

\subsubsection{Generalities about the $2$-adic square root}
We begin by noting the obvious fact that, whenever $x\equiv y\bmod{2^{k-1}}$, $x^2\equiv y^2\bmod{2^k}$ automatically. Thus, when taking square roots of $u\bmod{2^k}$, we should be looking at $x\bmod{2^{k-1}}$ such that $x^2\equiv u\bmod{2^k}$.

We claim that, for every $k\geqslant 3$, and for every $u\equiv 1\bmod 8$, there exist exactly two values of $x\bmod 2^{k-1}$ (which arise in two $2$-adic towers) such that
\[ x^2\equiv u\pmod{2^k}. \]
This follows by a standard Hensel's Lemma argument. The claim is clearly true for $k=3$: $x\equiv\pm 1\bmod 4$ satisfy $x^2\equiv 1\bmod 8$. Inductively, if, for some $k\geqslant 3$, we have a solution
\[ x_0^2\equiv u\pmod{2^k},\quad\text{that is,}\quad x_0^2=u+2^kY, \]
and we are looking to lift this to a solution $x_0+2^{k-1}t\pmod{2^k}$ ($t\in\{0,1\}$) satisfying
\[ (x_0+2^{k-1}t)^2\equiv u\pmod{2^{k+1}}, \]
in light of $2k-2\geqslant k+1$ this is equivalent to
\[ 2^kx_0t\equiv -2^kY\pmod{2^{k+1}},\quad t\equiv -Y\overline{x_0}\pmod 2, \]
and thus we have a unique lift $x_0+2^{k-1}t\pmod{2^k}$. The claim follows by induction.

We can now denote the two (exact) $2$-adic square roots by $\pm u_{1/2}(x)$ with $u_{1/2}(x)\equiv 1\bmod 4$, and we simply write $x_{1/2}$ for $u_{1/2}(x) \bmod{2^k}$ with a sufficiently large value of $k$. The claim indicates, for $k\geqslant 3$ and $x\equiv y\equiv 1 \bmod 8$,
\begin{align}\label{eqxx2}
 x\equiv y\pmod{2^{k}}\,\Leftrightarrow\,x_{1/2}\equiv y_{1/2}\pmod{2^{k-1}}.
\end{align}

\begin{lemma}\label{lemuu'}
If $u,u'\equiv 1\bmod 8$, $u\equiv u'\pmod{2^{\lambda+1}}$, $t\equiv t'\pmod{2^{\lambda}}$, and $k\geqslant 3$, then
\begin{align}\label{equu}
 (u+2^kt)_{1/2}-(u'+2^kt')_{1/2}\equiv u_{1/2}-u'_{1/2}+\overline{u_{1/2}}2^{k-1}t-\overline{u'_{1/2}}2^{k-1}t'  \pmod{2^{2k+\lambda-3}}.
\end{align}
\end{lemma}
\begin{proof}
The most direct way to prove this is by comparing
\[ (u+2^kt)_{1/2}\quad\text{and}\quad \sum_{j=0}^n\binom{1/2}{j}u_{1/2}^{1-2j}(2^kt)^j. \]
Indeed, using the agreement of formal power series
$((1+x)^{1/2})^2=1+x$,
we have that
\begin{align*}
&\ord_2\Big[\Big(\sum_{j=0}^n\binom{1/2}{j}u_{1/2}^{1-2j}(2^kt)^j\Big)^2 - (u+2^kt)\Big]\geqslant\min_{j_1+j_2\geqslant n+1}\ord_2\Big[\binom{1/2}{j_1}\binom{1/2}{j_2}(2^k)^{j_1+j_2}\Big]\\
&\qquad=\min_{j_1+j_2\geqslant n+1}\left[(k-1)(j_1+j_2)-\ord_2(j_1!)-\ord_2(j_2!)\right]\geqslant (n+1)(k-2)+2,
\end{align*}
where we use the obvious bound
$\ord_2(j!)=\lfloor j/2\rfloor+\lfloor j/4\rfloor+\dots \leqslant j-1$.
From this it follows that, for $k\geqslant 3$,
\[  (u+2^kt)_{1/2}\equiv\sum_{j=0}^n\binom{1/2}{j}u_{1/2}^{1-2j}(2^kt)^j\pmod{2^{(n+1)(k-2)+1}}. \]
From this we conclude that
\begin{align*}
&(u+2^kt)_{1/2}-(u'+2^kt')_{1/2}\\
&\qquad \equiv u_{1/2}-u'_{1/2}
+\sum_{j=1}^n\binom{1/2}j2^{kj}\l(u_{1/2}^{1-2j}t^j-u_{1/2}'{}^{1-2j}t'{}^j\r)\pmod{2^{(n+1)(k-2)+1}}.
\end{align*}
From $u\equiv u'\pmod{2^{\lambda+1}}$, we have that $u_{1/2}\equiv u'_{1/2}\pmod{2^{\lambda}}$, whence
\[ \ord_2\Big[\binom{1/2}j2^{kj}\l(u_{1/2}^{1-2j}t^j-u_{1/2}'{}^{1-2j}t'{}^j\r)\Big]\geqslant(k-2)j+\lambda+1. \]
Thus, choosing $n\geqslant\lambda+1$, we conclude that
\[ \ord_2\Big[(u+2^kt)_{1/2}-(u'+2^kt')_{1/2}-\l(u_{1/2}-u'_{1/2}\r)-\l(\overline{u_{1/2}}2^{k-1}t-\overline{u'_{1/2}}2^{k-1}t'\r)\Big]\geqslant 2k+\lambda-3, \]
as announced.
\end{proof}

\subsubsection{Kloosterman sum evaluation}
For $s\ge 6$ and $2\nmid x$, let
\begin{align}\label{eqtaux}
\tau(x,2^s):=
\begin{cases}
\frac1{2\sqrt2}\sum\limits_{t\bmod 4}e\l(\frac{xt^2}4\r)=\frac{1+e(x/4)}{\sqrt2}\in\{e^{\pm i\pi/4}\}, & 2\mid s,\\
\frac1{4}\sum\limits_{t\bmod 8}e\l(\frac{xt^2}8\r)=\frac{1+2e(x/8)+e(x/2)}{2}=e(x/8), & 2\nmid s,
\end{cases}
\end{align}
which depends only on the parity of $s$ and the value of $x \bmod 8$.

\begin{lemma}\label{lemKlo2}
For $s\ge 6$ and $2\nmid a$, then
$S(a,b;2^s)=0$ unless $a\equiv b\bmod 8$, in which case
\begin{align*}
S(a,b;2^s)=2^{(s+1)/2}\sum_{\pm}\tau(\pm (ab)_{1/2},2^s)e\Bigg(\frac{\pm 2(ab)_{1/2}}{2^s}\Bigg).
\end{align*}
\end{lemma}
\begin{proof}
We begin by noting the congruence
\[ \frac{b}{x+p^kt}\equiv \frac bx-\frac b{x^2}p^kt+\frac{b}{x^3}p^{2k}t^2\pmod{p^{3k}}, \]
which is valid for all $k\geqslant 1$ and irrespective of the value of $p$ (including $p=2$). Thus, $f(x)=ax+b\bar{x}$ satisfies
\[ f(x_0+p^kt)\equiv f(x_0)+\Bigg(a-\frac{b}{x_0^2}\Bigg)\cdot p^kt+\frac{b}{x_0^3}\cdot p^{2k}t^2\pmod{p^{3k}}. \]

For $s$ even, we have by the standard stationary phase argument (see \cite[Lemmas 12.2 \& 12.3]{IK04}) that
\[ S(a,b;2^s)=\ssum_{\substack{x\bmod 2^s\\a-b/x^2\equiv 0\bmod{2^{s/2}}}}e\Bigg(\frac{ax+b\bar{x}}{2^s}\Bigg). \]
In particular, for $s\geqslant 6$ even,
\[ S(m,n;2^s)=0\quad\text{unless}\quad a\equiv b\bmod{8}, \]
while in the case $a\equiv b\bmod 8$ the stationary phase congruence $a-b/x^2\equiv 0\pmod{2^{s/2}}$
is equivalent to $x\equiv\pm (b\bar{a})_{1/2}\pmod{2^{s/2-1}}$ and
\begin{align*}
S(a,b;2^s)&=\sum_{\pm}\sum_{t\bmod 2^{s/2+1}}e\Bigg(\frac{f\l(\pm (b\bar{a})_{1/2}+2^{s/2-1}t\r)}{2^s}\Bigg)\\
&=2^{s/2-1}\sum_{\pm}e\Bigg(\frac{\pm 2(ab)_{1/2}}{2^s}\Bigg)\sum_{t\bmod 4}e\Bigg(\frac{\pm a_{1/2}^3\overline{b_{1/2}}t^2}4\Bigg)\\
&=2^{(s+1)/2}\sum_{\pm}\tau(\pm(ab)_{1/2},2^s)e\Bigg(\frac{\pm 2(ab)_{1/2}}{2^s}\Bigg).
\end{align*}

For $s\geqslant 7$ odd, the same argument shows again that $S(a,b;2^s)=0$ unless $a\equiv b\bmod 8$, in which case $S(a,b;2^s)$ equals
\begin{align*}
\sum_{\substack{x\bmod 2^s\\a-b/x^2\equiv 0\bmod{2^{(s-1)/2}}}}e\Bigg(\frac{ax+b\bar{x}}{2^s}\Bigg)
&=\sum_{\pm}\sum_{t\bmod 2^{(s+1)/2+1}}e\Bigg(\frac{f\left(\pm (b\bar{a})_{1/2}+2^{(s-1)/2-1}t\right)}{2^s}\Bigg)\\
&=2^{(s+1)/2-2}\sum_{\pm}e\Bigg(\frac{\pm 2(ab)_{1/2}}{2^s}\Bigg)\sum_{t\bmod 8}e\Bigg(\frac{\pm a_{1/2}^3\overline{b_{1/2}}t^2}8\Bigg)\\
&=2^{(s+1)/2}\sum_{\pm}\tau(\pm (ab)_{1/2},2^s)e\Bigg(\frac{\pm 2(ab)_{1/2}}{2^s}\Bigg).
\end{align*}
This completes the proof.
\end{proof}

\subsubsection{Estimation of complete sums}
For $s\ge 2$ and $\epsilon\in \{\pm1\}$, we will need estimates on the sum
\[
\mathscr{S}^{\epsilon}_d(h,k_1,k_2,2^s):=\ssum_{m\ppmod {2^s}}S^{\epsilon}(k_1,md,2^s)\ol{S^{\epsilon}(k_2,md,2^s)}e\l(-\frac{hm}{2^s}\r),
\]
where
\begin{align}
S^{\epsilon}(m,n,2^s)=
\begin{cases}
2^{(s+1)/2}\tau\l(\epsilon(mn)_{1/2}, 2^s\r)e\l(\frac{2\epsilon(ab)_{1/2}}{2^s}\r),& m\equiv n \pmod 8,\\
0,& \text{otherwise}.
\end{cases}
\end{align}
\begin{lemma}\label{lemcs2}
For $2\nmid k_1k_2d$, $s\ge2$, and $\epsilon\in \{\pm1\}$, we have
\[
\mathscr{S}^{\epsilon}_d(h,k_1,k_2,2^s)\ll 2^{\frac{3s}{2}}\l(k_1-k_2,h,2^s\r)^{\frac12}\delta_{(k_1-k_2, 2^s)\mid  8h}.
\]
\end{lemma}
\begin{proof}
Since the estimate is trivial for small $s$, we may assume that $s\geqslant 4$.
We need to estimate
\begin{align*}
\mathscr{S}^{\epsilon}_d(h,k_1,k_2,2^s)=&2^{(s+1)}\sum_{\substack{m\ppmod{2^{s}}\\ m\equiv k_1\equiv k_2\ppmod 8}}\tau\l(\epsilon(mk_1)_{1/2}, 2^s\r)\ol{\tau\l(\epsilon(mk_2)_{1/2}, 2^s\r)}\\
&\times e\B(\frac{2\epsilon\l((mk_1)_{1/2}-(mk_2)_{1/2}\r)-\ol{d}hm}{2^s}\B).
\end{align*}
We write $(k_1-k_2, 2^s)=2^r$ with $r\ge3$, $t=\lfloor (s+r)/2 \rfloor-3$, and $m=m_1+2^{s-t}m_2$. Recall that $\tau(x,2^s)$ depends only on $x$ modulo $8$.
Applying Lemma \ref{lemuu'} with $\lambda=r-1$ and $k=s-t\ge 3$, we have
\begin{align*}
\mathscr{S}^{\epsilon}_d(h,k_1,k_2,2^s)=&2^{(s+1)}\sum_{\substack{m_1\ppmod{2^{s-t}}\\ m_1\equiv k_1\equiv k_2\ppmod 8}}\tau\l(\epsilon(m_1k_1)_{1/2}, 2^s\r)\ol{\tau\l(\epsilon(m_1k_2)_{1/2}, 2^s\r)}\\
&\times e\B(\frac{2\epsilon\l((m_1k_1)_{1/2}-(m_1k_2)_{1/2}\r)-\ol{d}hm_1}{2^s}\B)\\
&\times
\sum_{m_2\ppmod {2^t}} e\B(\frac{\epsilon m_2\l(k_1\ol{(m_1k_1)_{1/2}}-k_2\ol{(m_1k_2)_{1/2}}\r)-\ol{d}hm_2}{2^t}\B),
\end{align*}
observing $2k+\lambda-3=2s+r-2t-4>s$.
The $m_2$-sum vanishes unless
\begin{align}\label{eqsm0}
k_1\ol{(m_1k_1)_{1/2}}-k_2\ol{(m_1k_2)_{1/2}}-\epsilon\ol{d}h\equiv0 \pmod{2^t},
\end{align}
in which case it is equal to $2^t$.
Through \eqref{eqxx2}, the condition $(k_1-k_2, 2^s)=2^r$ yields
\begin{align}\label{eqk1k2}
\l(k_1\ol{(m_1k_1)_{1/2}}-k_2\ol{(m_1k_2)_{1/2}}, 2^s\r)=2^{r-1}.
\end{align}

For $3\le r\le s-6$, we have $r-1<t$,
and thus $2^{r-1}\parallel h$, derived from \eqref{eqsm0} and \eqref{eqk1k2}.
Writing $h=h'2^{r-1}$, we reduce \eqref{eqsm0} to
\begin{align}\label{eqsm1}
2^{-(r-1)}\l(k_1\ol{(m_1k_1)_{1/2}}-k_2\ol{(m_1k_2)_{1/2}}\r)\equiv\epsilon\ol{d}h' \pmod{2^{t-r+1}}.
\end{align}
By introducing a fixed $u\equiv k_1\bmod 8$, we note
\[
k_1\ol{(m_1k_1)_{1/2}}-k_2\ol{(m_1k_2)_{1/2}}\in \left\{\pm \l(\frac{(k_1u)_{1/2}}{(m_1u)_{1/2}}-\frac{(k_2u)_{1/2}}{(m_1u)_{1/2}}\r)\right\},
\]
and then reduce \eqref{eqsm1} to
\[
\ol{d}h'(m_1u)_{1/2}\equiv \pm 2^{-r+1}\l((k_1u)_{1/2}- (k_2u)_{1/2}\r) \pmod{2^{t-r+1}},
\]
whence, there are at most $O(1)$ solutions $m_1 \bmod 2^{t-r+1}$.
Therefore, there are at most $O(2^{s-2t+r-1})$ contributing values of $m_1$, whence
\[
\left|\mathscr{S}^{\epsilon}_d(h,k_1,k_2,2^s)\right|\ll 2^{s+1}2^{s-2t+r-1}2^t\ll 2^{\frac{3s}{2}}2^{\frac r2},
\]
in agreement with the statement of the Theorem.

For large $r\ge s-5$, we have $r-3\le t\le r-1$. Combining \eqref{eqsm0} and \eqref{eqk1k2} yields $2^t\mid h$, and then a trivial estimate shows
\[
\left|\mathscr{S}^{\epsilon}_d(h,k_1,k_2,2^s)\right|\ll 2^{2s}\ll 2^{\frac{3s}{2}}2^{\frac t2}\ll 2^{\frac{3s}{2}}\l(k_1-k_2,h,2^s\r)^{\frac12}.
\]
This completes the proof of the lemma.
\end{proof}

\subsection{Proof of Theorem \ref{lemSumMK}}
Since Lemmas \ref{lemKlo2} \& \ref{lemcs2} cover the $p=2$ case, unsettled in Lemma 22 and Proposition 23 of \cite{BM2015}, the argument of \cite{BM2015} then applies equally to these $q$ containing extremely high powers of $2$. This removes the condition $(q/s,2)=1$.

We next establish \eqref{eqSumMK} while retaining the coprimality condition $(m,q)=1$. The proof of this part largely follows \cite[Theorem 5]{BM2015}, with necessary adjustments to \cite[Section 9]{BM2015} detailed below.

To account for $\lambda(k)$, we decouple it from $b_1^\epsilon(k)$ and  adjust the definition \cite[(9.9)]{BM2015} to
\[
b_1(k)=\lambda(k)\sum_{\epsilon\in\{\pm1\}^\varrho}b_1^\epsilon(k),
\]
where $b_1^\epsilon(k)$ retains its original definition in \cite[(9.9)]{BM2015} but isolates $\lambda(k)$ as a separate factor.

The primary modification occurs at \cite[(9.12)]{BM2015}, which becomes
\begin{align*}
\Sigma&\ll r^\ve Hr_2^2\sum_{\epsilon\in\{\pm1\}^\varrho}\sum_{(m,q)=1}W\l(\frac mM\r)\sum_{\substack{K\le k_1,k_2\le 2K\\ k_1\equiv k_2 \bmod {Hr_2}\\ (k_1k_2,r_1^{\#})=1}}\lambda(k_1)\ol{\lambda(k_2)}b_1^\epsilon(k)\ol{b_1^\epsilon(k)},
\end{align*}
with
$r=r_1r_2$, $(r_1,r_2)=1$, $r_1^{\#}$ denoting the squarefull part of $r_1$, and $\text{rad}(r_1^{\#})$ its squarefree kernel, while $H$ is a parameter satisfying
\[
H\mid \frac{r_1^{\#}}{\text{rad}(r_1^{\#})}.
\]
Diagonal terms ($k_1=k_2$) contribute
\[
\ll  (rM)^\ve Hr_2^2Mr_1\sum_{K\le k\le 2K}|\lambda(k)|^2\ll (rKM)^\ve Hr_2^2MKr_1.
\]
By symmetry in $k_1$ and $k_2$, we apply
\[
\lambda(k_1)\ol{\lambda(k_2)}\ll |\lambda(k_1)|^2+|\lambda(k_2)|^2
\]
to off-diagonal terms. Employing M\"obius inversion to eliminate $(m,q)=1$ yields
\begin{align}\label{eq1.3}
\Sigma&\ll  (rKM)^\ve Hr_2^2MKr_1+r^\ve Hr_2^2 \sum_{d\mid q}\sum_{\substack{K\le k_1\neq k_2\le 2K\\ k_1\equiv k_2 \bmod {Hr_2}\\ (k_1k_2,r_1^{\#})=1}}|\lambda(k_1)|^2 \B|\sum_{\epsilon\in\{\pm1\}^\varrho}\Sigma_d^\epsilon(k_1,k_2,r_1)\B|,
\end{align}
where $\Sigma_d^\epsilon(k_1,k_2,r_1)$ remains as defined in \cite{BM2015}.

The second modification concerns the bound for $\sum_{\epsilon\in\{\pm1\}^\varrho}\Sigma_d^\epsilon(k_1,k_2,r_1)$. From the last formula in \cite[p.502]{BM2015}, it is shown that
\begin{align*}
\sum_{\epsilon\in\{\pm1\}^\varrho}\Sigma_d^\epsilon(k_1,k_2,r_1)\ll \sum_{\rho\mid \big(\frac{r_1}{r_1^{\#}},d\big)}\frac{(k_1,\rho)(k_2,\rho)}{\rho^{\frac12}} \frac{M}{d}r_1^{\frac12+\ve}(k_1-k_2,r_1)^{\frac12}\B(1+\frac{dr_1}{M\Big(k_1-k_2,\frac{r_1^{\#}}{\text{rad}(r_1^{\#})}\Big)\rho}\B).
\end{align*}
Using
\begin{align*}
(k_1,\rho)(k_2,\rho)=(k_1k_2,\rho)(k_1,k_2,\rho)\le (k_1k_2,\rho)(k_1-k_2,\rho)
\end{align*}
and $(\rho,H)=1$ since $(\rho, r_1^{\#})=1$, we have
\begin{align*}
\sum_{\epsilon\in\{\pm1\}^\varrho}\Sigma_d^\epsilon(k_1,k_2,r_1)\ll \sum_{\substack{\rho\mid (d,r_1)\\ (\rho,H)=1}}\frac{(k_1-k_2,\rho)}{\rho^{\frac12}} \B(Mr_1^{\frac12+\ve}(k_1-k_2,r_1)^{\frac12}+r_1^{\frac32+\ve}\frac{(k_1-k_2,r_1)^{\frac12}}{\Big(k_1-k_2,\frac{r_1^{\#}}{\text{rad}(r_1^{\#})}\Big)}\B).
\end{align*}
Noting $\rho\mid r_1$ and $(\rho,Hr_2)=1$,
we have
\begin{align*}
&\sum_{K\le k\le 2K}|\lambda(k)|^2 \sum_{0\neq l\ll \frac{K}{Hr_2}}\B|\sum_{\epsilon\in\{\pm1\}^\varrho}\Sigma_d^\epsilon(k,k+lHr_2,r_1)\B|\\
&\ll (dr_1)^\ve\sum_{K\le k\le 2K}|\lambda(k)|^2 \sum_{0\neq l\ll \frac{K}{Hr_2}} \sum_{\substack{\rho\mid (d,r_1)\\ (\rho,H)=1}}\frac{(l,\rho)}{\rho^{\frac12}}
\B(Mr_1^{\frac12}(l,r_1)^{\frac12}H^{\frac12}+r_1^{\frac32}\frac{(l,r_1)^{\frac12}H^{\frac12}}{H}\B)\\
&\ll (dr_1)^\ve\sum_{K\le k\le 2K}|\lambda(k)|^2 \sum_{0\neq l\ll \frac{K}{Hr_2}}
\B(Mr_1^{\frac12}(l,r_1)H^{\frac12}+r_1^{\frac32}\frac{(l,r_1)}{H^{\frac12}}\B)\\
&\ll (dr_1)^\ve\frac{K^2}{Hr_2}MH^{\frac12}r_1^{\frac12}+(dr_1)^\ve\frac{K^2}{Hr_2}\frac{r_1^{\frac32}}{H^{\frac12}}.
\end{align*}
Substituting this into \eqref{eq1.3} replicates the bound \cite[(9.14)]{BM2015} for $\Sigma$. The remainder of the proof aligns with \cite{BM2015}, yielding \eqref{eqSumMK} with the coprimality condition $(m,q)=1$ still present.

We finally eliminate $(m,q)=1$ through a detailed analysis of Kloosterman sums.
Let $d=(m,q)$, and for $p\mid d$ write $q=p^j q_p$. The twisted multiplicativity of Kloosterman sums gives
\[
\K(ckm;q)=\K(c\ol{q}_p^2km;p^j)\K(c\ol{p}^{2j}km;q_p).
\]
If $2\mid d$, Lemma \ref{lemKlo2} shows that the Kloosterman sum vanishes unless $q$ is divisible by only a small power of $2$ (not exceeding $2^5$). Since a small power of $2$ can be removed at no cost, we may assume $(d,2)=1$ in the following analysis.
For an odd prime $p\mid d$,
the Kloosterman sum $\K(c\ol{q}_p^2km;p^j)$ vanishes unless $j=1$, in which case, it reduces to a Ramanujan sum. In other words,
\begin{align*}
\K(c\ol{q}_p^2km;p^j)=\begin{cases}
-p^{-\frac12},& \text{if}\ j=1;\\
0,& \text{otherwise}.
\end{cases}
\end{align*}
Thus $\K(ckm;q)$ vanishes unless $q=dq_d$ with $d$ square-free, in which case,
\[
\K(ckm;q)=\frac{\mu(d)}{d^{\frac12}}\K(c\ol{d}^2km;q_d).
\]
Substituting this into \eqref{eqSumKlo} shows
\begin{align*}
\sum_{M\le m\le 2M}\B|\sum_{K\le k\le 2K}\lambda(k)\K(ckm;q)\B|^2&=\sum_{d\mid q}\frac{\mu^2(d)}{d}\sum_{\substack{M/d\le m\le 2M/d\\ (m,q_d)=1}}\B|\sum_{K\le k\le 2K}\lambda(k)\K(c\ol{d}km;q_d)\B|^2\\
&\ll (qKM)^\ve\|\lambda\|_2\B(Ms+\frac{KMs^{\frac12}}{q^{\frac12}}+\frac{Kq^{\frac12}}{s^{\frac12}}\B),
\end{align*}
which establishes the lemma.

\section{Bilinear forms for moduli with large prime factors}
From this section, we start our proof of Theorem \ref{thmBK}.

\subsection{Initial treatment for Theorem \ref{thmBK} with ``shifted by $ab$'' trick}
We expand the square in \eqref{eqSumKlo} to obtain
\begin{align}\label{eqSum}
\sum_{N\le n\le 2N}\B|\sum_{M\le m\le 2M}\alpha_m\K(cmn;q)\B|^2= \sum_{d\mid q}\mT_d(\bm{\alpha},1_{\mN}),
\end{align}
where
\begin{align*}
\mT_d(\bm{\alpha},1_{\mN})=\mathop{\sum\sum}_{\substack{M\le m_1,m_2\le 2M\\ (m_1-m_2,q)=d}}\alpha_{m_1}\ol{\alpha}_{m_2}\sum_{n\in\mN}\K(cm_1n;q)\K(cm_2n;q).
\end{align*}
A trivial estimate with the inequality $\alpha_{m_1}\ol{\alpha}_{m_2}\ll |\alpha_{m_1}|^2+|\alpha_{m_2}|^2$ shows
\begin{align}\label{eqtbT}
\mT_d(\bm{\alpha},1_{\mN})\ll \l(1+\frac{M}d\r)N\|\bm{\alpha}\|_2^2,
\end{align}
which will be employed when $(d,q^\star)>1$. For the remaining case $d\mid \rho$, we apply Vinogradov's ``shifted by $ab$'' summation approach and analyze it through three distinct scenarios:
\begin{itemize}
  \item $q^{\star}=p$ is a prime;
  \item $q^{\star}=p_1p_2$ is a product of two distinct primes;
  \item $q^{\star}=p^2$ is a square of a prime.
\end{itemize}

For given $A,B\ge1$ satisfying
\[
2B<q,\quad AB\le N,\quad AM<q,
\]
we have
\begin{align*}
\mT_d(\bm{\alpha},1_{\mN})=&\frac{1}{A^*B}\mathop{\ssum\sum}_{\substack{A\le a\le 2A\\ B\le b\le 2B}}\mathop{\sum\sum}_{\substack{M\le m_1,m_2\le 2M\\ (m_1-m_2,q)=d}}\alpha_{m_1}\ol{\alpha}_{m_2}\\
&\times\sum_{n+ab\in\mN}\K(acm_1(\ol{a}n+b);q)\K(acm_2(\ol{a}n+b);q),
\end{align*}
where $A^*$ counts the integers $a\in[A,2A]$ coprime to $q$. By performing the variable substitutions
\[
am_1\rightarrow l_1,\quad am_2\rightarrow l_2,\quad \ol{a}n\rightarrow r
\]
and introducing a smooth weight function $W(x)\ge0$, which has compact support in $[1/2,5]$ with $W^{j}(x)\ll 1$ for $j\ge0$ and takes the value of $1$ in $[1,4]$,
we obtain via the method of \cite[p. 116]{FM98} that
\begin{align*}
\mT_d(\bm{\alpha},1_{\mN})\ll \frac{q^\ve}{AB}\sum_{r\ppmod q}&\mathop{\sum\sum}_{(l_1-l_2,q)=d}\nu(r,l_1,l_2)W\l(\frac{l_1}{AM}\r)W\l(\frac{l_2}{AM}\r)\\
&\times\bigg|\sum_{B\le b\le 2B}\gamma_b\K(cl_1(r+b);q)\K(cl_2(r+b);q)\bigg|,
\end{align*}
where the coefficient $\nu(r,l_1,l_2)$ aggregates contributions from
\[
\nu(r,l_1,l_2)=\mathop{\ssum\sum\cdots\sum}_{\substack{A\le a\le 2A,\ M\le m_1,m_2\le 2M,\ n\in\mN'\\ am_1=l_1,\ am_2=l_2,\ \ol{a}n\equiv r\ppmod q}}|\alpha_{m_1}||\alpha_{m_2}|.
\]
Here $\mN'\supset\mN$ denotes an extended interval of length $2N$, and $\gamma_b$ are some complex coefficients with $|\gamma_b|\le1$. The following bounds hold
\[
\sum_r\mathop{\sum\sum}_{(l_1-l_2,q)=d}\nu(r,l_1,l_2)W\l(\frac{l_1}{AM}\r)W\l(\frac{l_2}{AM}\r)\ll AN\|\bm{\alpha}\|_1^2\ll AMN\|\bm{\alpha}\|^2_2
\]
and
\begin{align*}
\sum_r\mathop{\sum\sum}_{(l_1-l_2,q)=d}\nu(r,l_1,l_2)^2W\l(\frac{l_1}{AM}\r)W\l(\frac{l_2}{AM}\r)&\ll \mathop{\sum\cdots\sum}_{\substack{am_1=a'm_1'\\ am_2=a'm_2'\\ an\equiv a'n'\ppmod q}}\left|\alpha_{m_1}\alpha_{m'_1}\alpha_{m_2}\alpha_{m'_2}\right|\\
&\ll q^\ve\sum_{a,m_1,m_2,n}|\alpha_{m_1}|^2|\alpha_{m_2}|^2\ll q^\ve AN\|\bm{\alpha}\|_2^4,
\end{align*}
since the cross terms satisfy
\[
\left|\alpha_{m_1}\alpha_{m'_1}\alpha_{m_2}\alpha_{m'_2}\right|\le \l(|\alpha_{m_1}|^2+|\alpha_{m'_1}|^2\r)\l(|\alpha_{m_2}|^2+|\alpha_{m'_2}|^2\r)
\]
and $m_i, m_i'$ are determined by each other up to $O(q^\ve)$ possibilities.
The congruence $a'n\equiv an'$ restricts $n'$ to at most two values in $\mN'$ due to its length $\le 2q$.
Applying the H\"{o}lder inequality yields
\begin{align}\label{eqT}
\mT_d(\bm{\alpha},1_{\mN})&\ll \frac{q^\ve}{AB}(AN)^{\frac34}M^{\frac12}\|\bm{\alpha}\|^{2}_2\bigg(\sum_{\bm{b}}\left|\Sigma^{d}(\bm{b},AM;q)\right|\bigg)^{\frac14},
\end{align}
where the key exponential sum is
\begin{align}\label{eqSigma}
\Sigma^{d}(\bm{b},AM;q)=\sum_{r\ppmod q}\mathop{\sum\sum}_{(l_1-l_2,q)=d}W\l(\frac{l_1}{AM}\r)W\l(\frac{l_2}{AM}\r)\prod_{i=1}^2\prod_{j=1}^4\K(cl_i(r+b_j);q).
\end{align}

\begin{remark}
The Vinogradov's ``shifted by $ab$'' technique employed here constitutes a refinement of methods established in prior research \cite{BFK+17a} and \cite{FM98}. A novel element lies in the incorporation of a smooth test function, which facilitates the treatment of the exponential sum via Poisson summation formula rather than the discrete Plancherel identity in subsequent section. This approach enables potential savings in the frequency sum estimation.
\end{remark}

\subsection{Reduction to  multiplicative completed exponential sums}
Let $\mathbb{B}$ denote the set of all $4$-tuples satisfying
\[
\mathbb{B}=\{b=(b_1,b_2,b_3,b_4): 1\le b_i\le B\}.
\]
For a prime $p$, define $\mV^\Delta$ as the affine variety of $4$-tuples
\[
\bm{b}=(b_1,b_2,b_3,b_4)\in \bbm{F}_p^4
\]
characterized by the condition that for each index $i\in\{1,2,3,4\}$, the cardinality
\[
|\{j=1,2,3,4\ |\ b_j=b_i\}|
\]
must be even. We further define $\mathbb{B}^\Delta_p$ as the subset of integer tuples $\bm{b}\in\mathbb{B}$ satisfying
\[
\bm{b}\ppmod p\in \mV^\Delta.
\]

A direct computation yields the trivial bound
\begin{align}\label{eqTBSigma}
\Sigma^{d}(\bm{b},AM;q)\ll A^2M^2 q
\end{align}
for arbitrary $\bm{b}\in\mathbb{B}$ and divisors $d\mid q$.

To obtain refined estimates for $\Sigma^{d}(\bm{b},AM;q)$, we employ distinct analytical approaches. The primary method involves evaluating the $r$-summation  for fixed $(l_1,l_2)$ pairs. For $\bm{l}=(l_1,l_2)$, we define the completed exponential sum
\[
\mR(\bm{b},\bm{l};q)=\sum_{r\ppmod q}\prod_{i=1}^2\prod_{j=1}^4\K(l_i(r+b_j);q).
\]
Through the twisted multiplicativity of Kloosterman sums and application of the Chinese remainder theorem, $\mR(\bm{b},\bm{l};q)$ inherits the following multiplicative property.
\begin{lemma}[Multiplicative property]\label{lemmMR}
Let $q=q_1q_2$ with $(q_1,q_2)=1$. Then for arbitrary $\bm{b}$ and $\bm{l}$, we have
\[
\mR(\bm{b},\bm{l};q)=\mR(\bm{b},\ol{q}^2_2\bm{l};q_1)\mR(\bm{b},\ol{q}^2_1\bm{l};q_2).
\]
\end{lemma}
\begin{proof}
The twisted multiplicativity of Kloosterman sums implies the decomposition
\[
\K(l_i(r+b_j);q)=\K(l_i(r+b_j)\ol{q}_2^2;q_1)\K(l_i(r+b_j)\ol{q}_1^2;q_2).
\]
The claimed identity then follows from the Chinese remainder theorem.
\end{proof}

A direct computation yields the uniform trivial bound
\begin{align}\label{eqTBR}
\mR(\bm{b},\bm{l};q)\ll q,
\end{align}
valid for all $\bm{b}$ and $\bm{l}$.

\begin{lemma}\label{lemR1}
For tuples $\bm{b}\notin\mathbb{B}_p^\Delta$ and $\bm{l}=(l_1,l_2)$ satisfying $l_1\not\equiv l_2 \ (\bmod \ {p})$, we have
\begin{align}\label{eqR1}
\mR(\bm{b},\bm{l};p)\ll p^{\frac12}.
\end{align}
\end{lemma}
\begin{proof}
When $l_1l_2\not\equiv0\ (\bmod \ {p})$, the result follows from \cite[Lemma 2.5]{KMS2017}. In the remaining cases where some $l_i$ vanishes in $\bbm{F}_p$, the Kloosterman sum reduces to a Ramanujan sum, such that
\[
\K(l_i(r+b_j);p)=-p^{-\frac12},
\]
The bound \eqref{eqR1} then follows by estimating the remaining terms trivially.
\end{proof}

The subsequent analysis requires more sophisticated techniques to extract savings from averaging over $l_1$ and $l_2$. By applying the Poisson summation formula to transform this average into completed sums, we obtain
\begin{align}\label{eqSS}
\Sigma^{d}(\bm{b},AM;q)\ll \frac{(AM)^2}{q^2}\sum_{\bm{h}}\left|\mathfrak{S}(\bm{b},\ol{c}\bm{h},d;q)\right|,
\end{align}
where the completed exponential sum is defined as
\begin{align}\label{eqdefS}
\mathfrak{S}(\bm{b},\bm{h},d;q)=\mathop{\sum\sum\sum}_{\substack{r,s_1,s_2\ppmod q\\(s_1-s_2,q)=d}}\prod_{i=1}^2\prod_{j=1}^4 e_q(h_is_i)\K\left(s_i(r+b_j);q\right)
\end{align}
with frequency parameters $\bm{h}=(h_1,h_2)$ satisfying
\[
h_i\ll H=\l(\frac {q}{AM}\r)^{1+\ve}.
\]
This completed exponential sum possesses the following multiplicative property.
\begin{lemma}[Multiplicative property]\label{lemmS}
Let $q=q_1q_2$ with $(q_1,q_2)=1$ and $d=d_1d_2$ where $d_i=(d,q_i)$. Then for arbitrary $\bm{b}$ and $\bm{h}$, we have
\[
\mathfrak{S}(\bm{b},\bm{h},d;q)=\mathfrak{S}(\bm{b},q_2\bm{h},d_1;q_1)\mathfrak{S}(\bm{b},q_1\bm{h},d_2;q_2).
\]
\end{lemma}
\begin{proof}
By the twisted multiplicativity of Kloostermann sums, the reciprocity law of exponent functions and the Chinese remainder theorem, we have
\begin{align*}
\mathfrak{S}(\bm{b},\bm{h},d;q)=&\mathop{\sum\sum\sum}_{\substack{r_1,s_1,s_2\ppmod {q_1}\\(s_1-s_2,q_1)=d_1}}\prod_{i=1}^2\prod_{j=1}^4 e_{q_1}(h_is_i\ol{q}_2)\K\left(s_i(r_1+b_j)\ol{q}_2^2;q_1\right)\\
&\times\mathop{\sum\sum\sum}_{\substack{r_2,t_1,t_2\ppmod {q_2}\\(t_1-t_2,q_2)=d_2}}\prod_{i=1}^2\prod_{j=1}^4 e_{q_2}(h_it_i\ol{q}_1)\K\left(t_i(r_2+b_j)\ol{q}_1^2;q_2\right).
\end{align*}
After making the variable substitutions
\[
s_i\ol{q}^2_2\rightarrow s_i,\quad t_i\ol{q}^2_1\rightarrow t_i,
\]
we obtain the multiplicative relation.
\end{proof}

For $d\mid \rho$, Lemma \ref{lemmS} yields the factorization
\begin{align}\label{eqSSS}
\mathfrak{S}(\bm{b},\bm{h},d;q)=\mathfrak{S}(\bm{b},q^\star\bm{h},d;\rho )\mathfrak{S}(\bm{b},\rho\bm{h},1;q^\star).
\end{align}
The first component admits the trivial bound
\begin{align}\label{eqtrivialS}
\mathfrak{S}(\bm{b},q^\star\bm{h},d;\rho )\ll  \rho^3,
\end{align}
while the evaluation of $\mathfrak{S}(\bm{b},\bm{h},1;q^\star)$ reduces, via the multiplicative relation, to analyzing the prime-power cases $\mathfrak{S}(\bm{b},\bm{h},1;p)$ and $\mathfrak{S}(\bm{b},\bm{h},1;p^2)$.

\begin{lemma}\label{lemS11}
For $\bm{b}\notin\mathbb{B}_p^\Delta$, we have
\begin{align}\label{eqBS11}
\mathfrak{S}(\bm{b},\bm{h},1;p)\ll p^{\frac52}.
\end{align}
\end{lemma}
\begin{proof}
We begin by observing the preliminary bound
\[
\mathfrak{S}(\bm{b},\bm{h},1;p)\ll\mathop{\sum\sum\sum}_{\substack{s_1,s_2\ppmod p\\(s_1-s_2,p)=1}}\left|\mR(\bm{b},\bm{s};p)\right|,
\]
where $\bm{s}=(s_1,s_2)$. For each fixed pair $\bm{s}$,  Lemma \ref{lemR1} establishes $\mR(\bm{b},\bm{s};p)\ll p^{\frac12}$ under the hypothesis $\bm{b}\notin\mathbb{B}_p^\Delta$. Then the trivial estimate for the double summation over $s_1$, $s_2$ gives the required bound.
\end{proof}

\begin{lemma}\label{lemS1}
For sufficiently large prime $p$, there exists a codimension $1$ subvariety $\mV^{bad}\subset\bbm{F}_p^4$ containing $\mV^\Delta$,  with degree bounded independently of $p$, such that for all $\bm{b}$ satisfying $\bm{b}\ (\bmod p)\notin\mV^{bad}$ and arbitrary $\bm{h}$, we have
\begin{align}\label{eqBS1}
\mathfrak{S}(\bm{b},\bm{h},1;p)\ll p^{\frac32}.
\end{align}
\end{lemma}
\begin{proof}
If we restrict the sum in the definition \eqref{eqdefS} to the terms with $s_1s_2\neq 0$, then the bound follows from \cite[Theorem 2.6]{KMS2017}, combining with the argument just before it.
For the remaining terms, in which one of $s_i=0$, the corresponding four Kloosterman sums reduce to
\[
\K(s_i(r+b_j);p)=-p^{-\frac12},
\]
which shows that their contribution is at most $O(p)$.
\end{proof}

\begin{theorem}\label{lemBS2}
For $\bm{b}\notin\mathbb{B}_p^\Delta$, we have
\begin{align}\label{eqBS2}
\mathfrak{S}(\bm{b},\bm{h},1;p^2)\ll
\begin{cases}
p^4 & \text{if}\ \bm{h}\ppmod p\equiv(0,0),\\
p^4 & \text{if}\  (\bm{b},\bm{h})\in \mV_4^{bad}\times\mV_2^{bad}(\bm{b}),\\
p^3 & \text{otherwise}.
\end{cases}
\end{align}
Here $\mV_4^{bad}\subset\mathbb{F}_p^4$ is a variety defined by a homogeneous polynomial of bounded degree, while for any given $\bm{b}$, the variety $\mV_2^{bad}(\bm{b})\subset\mathbb{F}_p^2$ is defined by
$O(1)$ linear homogeneous polynomials.
\end{theorem}

\begin{remark}
Note that bounds in \eqref{eqtrivialS} -- \eqref{eqBS1} are independent of the parameter vector $\bm{h}$, whereas the estimates in \eqref{eqBS2} exhibit homogeneous dependence on $\bm{h}$. Consequently, these bounds remain valid under the substitution $\bm{h}\rightarrow c\bm{h}$ for any integer $c$ satisfying $(c, p)=1$. This substitution principle is particularly significant when applying the multiplicative decomposition from Lemma \ref{lemmS}, as it ensures the preservation of estimate quality under scaling transformations of the frequency parameters. The uniform bounds thus obtained are essential for handling the exponential sums that arise in the subsequent analytic arguments.
\end{remark}

We defer the proof of Theorem \ref{lemBS2} to the following section and proceed to establish Theorem \ref{thmBK} assuming the validity of Theorem \ref{lemBS2}.
To begin with, we define the required notation.
For large prime $p$, we define the exceptional set
\[
\mathbb{B}_p^{bad}=\{\bm{b}\in\mathbb{B}\ |\ \bm{b}\ppmod p\in\mV^{bad}\}\setminus\mathbb{B}_p^\Delta,
\]
where $\mV^{bad}$ is as defined in Lemma \ref{lemS1}. This allows us to partition the parameter space into three disjoint subsets such that
\[
\mathbb{B}=\mathbb{B}_p^{\Delta}\cup\mathbb{B}_p^{bad}\cup\mathbb{B}_p^{gen}.
\]
Since $\mV^{bad}$  has a finite degree, independently of $p$, we have
\[
\left|\mathbb{B}_p^{\Delta}\right|\ll B^2\l(1+\frac {B^2}{p^2}\r),\quad \left|\mathbb{B}_p^{bad}\right|\ll B^3\l(1+\frac {B}{p}\r), \quad \left|\mathbb{B}^{gen}_p\right|\ll B^4.
\]

The proof of Theorem \ref{thmBK} will be carried out in three cases depending on the factorization of $q^\star$. In all cases, we employ the following parameter choices
\begin{align}\label{eqAB1}
A=M^{-\frac12}N^{\frac12}q^{\frac18}\rho^{\frac38}, \quad B=M^{\frac12}N^{\frac12}q^{-\frac18}\rho^{-\frac38}.
\end{align}
For $MN\le q^{\frac54}\rho^{-\frac14}$, it is easy to check the inequalities
\begin{equation}\label{eqAB2}
AM<q, \quad AB=N,\quad B\le q^{\frac12}\rho^{-\frac12}=(q^\star)^{\frac12},
\end{equation}
which we present here for easy of reference.

\subsubsection{Proof of Theorem \ref{thmBK} for $q^\star=p$}
Through the multiplicative property (Lemma \ref{lemmMR}), we derive from \eqref{eqSigma} that
\begin{align}\label{eqMSigma1}
\Sigma^{d}(\bm{b},AM;q)\ll (AM)^2\left|\mR(\bm{b},\ol{p}^2\bm{l};\rho)\mR(\bm{b},\ol{\rho}^2\bm{l};p)\right|.
\end{align}
For $\bm{b}\notin\mathbb{B}^{\Delta}_p$, we apply the trivial bound \eqref{eqTBR} to the first exponential sum and \eqref{eqR1} to the second one, obtaining
\begin{align}\label{1BSigma1}
\Sigma^{d}(\bm{b},AM;q)\ll (AM)^2(q\rho )^{\frac12}.
\end{align}

On the other hand, applying the multiplicative property (Lemma \ref{lemmS}) into \eqref{eqSS}, we arrive at
\begin{align}\label{eqMSigma2}
\Sigma^{d}(\bm{b},AM;q)\ll \frac{(AM)^2}{q^2}\sum_{\bm{h}}\left|\mathfrak{S}(\bm{b},p\bm{h},d;\rho )\mathfrak{S}(\bm{b},\rho\bm{h},1;p)\right|.
\end{align}
For $\bm{b}\in\mathbb{B}^{gen}_p$, we bound the first exponential sum via \eqref{eqtrivialS} and the second one via \eqref{eqBS1}, getting
\begin{align}\label{1BSigma2}
\Sigma^{d}(\bm{b},AM;q)\ll q^\ve (q\rho )^{\frac32}.
\end{align}

To combine all cases, we apply the trivial bound \eqref{eqTBSigma} for $\bm{b}\in\mathbb{B}^\Delta_p$, the bound \eqref{1BSigma1} for $\bm{b}\in\mathbb{B}^{bad}_p$ and finally the bound \eqref{1BSigma2} for $\bm{b}\in\mathbb{B}^{gen}_p$. For $p=q^\star\ge B$ (see \eqref{eqAB2}), this gives
\begin{align}\label{eqSumbS}
\sum_{\bm{b}}\left|\Sigma^{d}(\bm{b},AM;q)\right|\ll q^\ve\l(A^2B^2M^2q+A^2B^3M^2( q\rho)^{\frac12}+B^4( q\rho)^{\frac32}\r).
\end{align}
By using parameters $A$, $B$ from \eqref{eqAB1} under $MN\le q^{\frac54}\rho^{-\frac14}$, the first and third terms on the right-hand side of \eqref{eqSumbS} are equal to $(MN)^2q$, and the second term is
\[
(MN)^{\frac52}q^{\frac38}\rho^{\frac18}\ll (MN)^2q.
\]
We conclude
\begin{align*}
\sum_{\bm{b}}\left|\Sigma^{d}(\bm{b},AM;q)\right|\ll (MN)^2q^{1+\ve},
\end{align*}
and consequently by \eqref{eqT}
\begin{equation}\label{eqTd}
\begin{aligned}
\mT_d(\bm{\alpha},1_{\mN})&\ll \frac{q^\ve}{N}(AN)^{\frac34}M^{\frac12}\|\bm{\alpha}\|^{2}_2(MN)^{\frac12}q^{\frac14}\\
&=\|\bm{\alpha}\|^{2}_2(MN)^{\frac58}q^{\frac{11}{32}+\ve}\rho^{\frac9{32}}.
\end{aligned}
\end{equation}
Substituting this into \eqref{eqSum} with the trivial bound \eqref{eqtbT} for $p\mid d$, we arrive at
\begin{align*}
\sum_{N\le n\le 2N}\B|\sum_{M\le m\le 2M}\alpha_m\K(cmn;q)\B|^2\ll
q^\ve\|\bm{\alpha}\|^{2}_2MN\l(M^{-1}+p^{-1}+(MN)^{-\frac38}q^{\frac{11}{32}}\rho^{\frac{9}{32}}\r),
\end{align*}
which establishes \eqref{eqBKrho} through \eqref{eqthm5BM}.

\subsubsection{Proof of Theorem \ref{thmBK} for $q^\star=p_1p_2$, $p_1<p_2$}
In this case, we decompose $\mathbb{B}$ into seven disjoint subsets such that
\[
\mathbb{B}=\mathbb{B}^{\Delta}\cup\mathbb{B}^{bad}\cup\mathbb{B}^{gen}\cup\mathbb{B}_1\cup \mathbb{B}_2\cup\mathbb{B}'_1\cup \mathbb{B}'_2,
\]
where the components are defined as
\begin{align*}
\mathbb{B}^{\Delta}=\mathbb{B}^{\Delta}_{p_1}\cap\mathbb{B}^{\Delta}_{p_2}, \quad \mathbb{B}^{bad}=\mathbb{B}^{bad}_{p_1}\cap\mathbb{B}^{bad}_{p_2},  \quad \mathbb{B}^{gen}=\mathbb{B}^{gen}_{p_1}\cap\mathbb{B}^{gen}_{p_2},
\end{align*}
and
\begin{align*}
\mathbb{B}_1=\mathbb{B}^{\Delta}_{p_1}\setminus\mathbb{B}^{\Delta}_{p_2}, \quad \mathbb{B}_2=\mathbb{B}^{\Delta}_{p_2}\setminus\mathbb{B}^{\Delta}_{p_1},\quad
\mathbb{B}'_1=\l(\mathbb{B}^{bad}_{p_1}\cap\mathbb{B}^{gen}_{p_2}\r),\quad \mathbb{B}'_2=\l(\mathbb{B}^{bad}_{p_2}\cap\mathbb{B}^{gen}_{p_1}\r).
\end{align*}
Note $\mathbb{B}_j=\emptyset$ when $p_j>B$.
For $p_1p_2=q^\star\ge B$, it is easy to see
\[
\left|\mathbb{B}^{\Delta}\right|\ll B^2,\quad \left|\mathbb{B}^{bad}\right|\ll B^3\l(1+\frac{B}{p_2}\r), \quad \left|\mathbb{B}^{gen}\right|\ll B^4,
\]
and
\[
\left|\mathbb{B}_j\right|\ll \frac{B^4}{p_j^2}, \quad \left|\mathbb{B}'_j\right|\ll B^3\l(1+\frac{B}{p_j}\r).
\]

Applying Lemmas \ref{lemmMR} \& \ref{lemmS} yields the analogues of \eqref{eqMSigma1} and \eqref{eqMSigma2} such that
\begin{align}\label{eq2MSigma1}
\Sigma^{d}(\bm{b},AM;q)\ll (AM)^2\rho\left|\mR(\bm{b},(\ol{\rho p_2})^2\bm{l};p_1)\mR(\bm{b},(\ol{\rho p_1})^2\bm{l};p_2)\right|,
\end{align}
\begin{align}\label{eq2MSigma2}
\Sigma^{d}(\bm{b},AM;q)\ll \frac{(AM)^2}{q^2}\rho^{3}\sum_{\bm{h}}\left|\mathfrak{S}(\bm{b},p_2\rho\bm{h},d;p_1) \mathfrak{S}(\bm{b},p_1\rho\bm{h},1;p_2)\right|.
\end{align}

For $\bm{b}\in \mathbb{B}^{bad}$, applying \eqref{eqR1} to both exponential sums in \eqref{eq2MSigma1} yields the estimate \eqref{1BSigma1}.

For $\bm{b}\in \mathbb{B}^{gen}$, applying \eqref{eqBS1} to both exponential sums in \eqref{eq2MSigma2} reproduces the estimate \eqref{1BSigma2}.

For $\bm{b}\in \mathbb{B}_j$,  we employ the trivial bound for the exponential sum modulo $p_j$  in \eqref{eq2MSigma1} and \eqref{eqR1} for the other sum, which gives
\begin{align}\label{eqB1}
\Sigma^{d}(\bm{b},AM;q)\ll (AM)^2(q\rho )^{\frac12}p_j^{\frac12}.
\end{align}

For $\bm{b}\in \mathbb{B}'_j$, we apply \eqref{eqBS11} to the exponential sum modulo $p_j$ in \eqref{eq2MSigma2} and \eqref{eqBS1} to the other sum, which shows
\[
\Sigma^{d}(\bm{b},AM;q)\ll q^\ve(q\rho )^{\frac32}p_j.
\]
Combining this with \eqref{1BSigma1} yields
\begin{align}\label{eqB2}
\Sigma^{d}(\bm{b},AM;q)\ll q^\ve \min\left\{(AM)^2(q\rho )^{\frac12},\ (q\rho )^{\frac32}p_j\right\} .
\end{align}

To consolidate all cases, we apply the trivial bound \eqref{eqTBSigma} for $\bm{b}\in\mathbb{B}^\Delta$, the bound \eqref{1BSigma1} for $\bm{b}\in\mathbb{B}^{bad}$, the bound \eqref{1BSigma2} for $\bm{b}\in\mathbb{B}^{gen}$, the bound \eqref{eqB1} for $\bm{b}\in \mathbb{B}_j$, and finally the bound \eqref{eqB2} for $\bm{b}\in \mathbb{B}'_j$. This yields the comprehensive estimate
\begin{align*}
\sum_{\bm{b}}&\left|\Sigma^{d}(\bm{b},AM;q)\right|\ll q^\ve\Bigg\{A^2B^2M^2q+A^2B^3M^2( q\rho)^{\frac12}\l(1+\frac{B}{p_2}\r)+B^4( q\rho)^{\frac32}\\
&\ \ \ \ +\sum_{j=1,2}A^2B^4M^2( q\rho)^{\frac12}p_j^{-\frac32}
+\sum_{j=1,2}\min\left\{A^2B^3M^2(q\rho )^{\frac12},\ B^3(q\rho )^{\frac32}p_j\right\}\l(1+\frac{B}{p_j}\r)\Bigg\}.
\end{align*}
For $ p_2>(q/\rho)^{\frac12}>p_1>(q/\rho)^{\frac13}$, it is easy to see that
\[
A^2B^3M^2(q\rho)^{\frac12}\frac{B}{p_2}+\sum_{j=1,2}A^2B^4M^2( q\rho)^{\frac12}p_j^{-\frac32}\ll A^2B^4M^2\rho,
\]
\[
\sum_{j=1,2}\min\left\{A^2B^3M^2(q\rho )^{\frac12},\ B^3(q\rho )^{\frac32}p_j\right\}\l(1+\frac{B}{p_j}\r)\ll A^2B^3M^2(q\rho )^{\frac12}+B^4( q\rho)^{\frac32}.
\]
Consequently, we conclude that
\begin{align}\label{eqSumbS2}
\sum_{\bm{b}}\left|\Sigma^{d}(\bm{b},AM;q)\right|\ll q^\ve\Bigg\{A^2B^2M^2q+A^2B^3M^2( q\rho)^{\frac12}+B^4( q\rho)^{\frac32}+A^2B^4M^2\rho\Bigg\},
\end{align}
where all terms except the last also appear in \eqref{eqSumbS}.
With $A$, $B$ from \eqref{eqAB1} and $MN\ll q^{\frac54}\rho^{-\frac14}$, the last term satisfies
\[
(MN)^{3}q^{-\frac14}\rho^{\frac14}\ll (MN)^2q.
\]
Now we also arrive at
\begin{align*}
\sum_{\bm{b}}\left|\Sigma^{d}(\bm{b},AM;q)\right|\ll (MN)^2q^{1+\ve},
\end{align*}
reconfirming \eqref{eqBKrho} through identical arguments as before.

\subsubsection{Proof of Theorem \ref{thmBK} for  $q^{\star}=p^2$}
Adopting notations from Theorem \ref{lemBS2}, we define
\[
\mathbb{B}^{bad}=\{\bm{b}\in\mathbb{B}\ |\ \bm{b}\ppmod p\in\mV_4^{bad}\}\setminus\mathbb{B}_p^\Delta,\quad \mathbb{H}^{bad}=\{\bm{h}\in\mathbb{H}\ |\ \bm{h}\ppmod p\in\mV_2^{bad}(\bm{b})\}.
\]
Given the finite degrees of $\mV_2^{bad}(\bm{b})$ and $\mV_4^{bad}$, we have
\[
\left|\mathbb{B}^{bad}\right|\ll B^3\l(1+\frac{B}{p}\r),\quad \left|\mathbb{H}^{bad}\right|\ll H\l(1+\frac{H}{p}\r).
\]
Applying Lemma \ref{lemmS}'s multiplicative relation to \eqref{eqSS} yields
\[
\Sigma^{d}(\bm{b},AM;q)\ll \frac{(AM)^2}{q^2}\rho^{3}\sum_{\bm{h}}\left|\mathfrak{S}(\bm{b},\rho \ol{c}\bm{h},1;p^2)\right|.
\]
Bounding the exponential sum via Theorem \ref{lemBS2} gives
\begin{align*}
\sum_{\bm{b}\in\mathbb{B}\setminus\mathbb{B}_p^\Delta}\Sigma^{d}(\bm{b},AM;q)&\ll \frac{(AM)^2}{q^2}\rho^3\l(B^4p^4\l(1+\frac {H^2}{p^2}\r)+B^3Hp^4\l(1+\frac Bp\r)\l(1+\frac Hp\r)+B^4H^2p^3\r)\\
&\ll \frac{(AM)^2}{q^2}\rho^3\l(B^4p^4+B^3Hp^4+B^4H^2p^3\r)\\
&\ll q^\ve\l(A^2B^4M^2\rho +AB^3M q\rho+B^4( q\rho)^{\frac32}\r).
\end{align*}
Combining this with the trivial bound \eqref{eqTBSigma} for $\bm{b}\in\mathbb{B}^\Delta_p$ produces
\begin{equation}\label{eqSumbS3}
\begin{aligned}
\sum_{\bm{b}}\left|\Sigma^{d}(\bm{b},AM;q)\right|&\ll q^\ve\l(A^2B^2M^2q\l(1+\frac{B^2}{p^2}\r)+A^2B^4M^2\rho +AB^3M q\rho+B^4( q\rho)^{\frac32}\r)\\
&\ll q^\ve\l(A^2B^2M^2q+A^2B^4M^2\rho +B^4( q\rho)^{\frac32}\r),
\end{aligned}
\end{equation}
where $AB^3M q\rho$ is absorbed by the first and last terms. Noting that all terms in \eqref{eqSumbS3} appear in \eqref{eqSumbS2}, we also have
\begin{align*}
\sum_{\bm{b}}\left|\Sigma^{d}(\bm{b},AM;q)\right|\ll (MN)^2q^{1+\ve},
\end{align*}
available for $A$, $B$ in \eqref{eqAB1}. This establishes \eqref{eqBKrho}, concluding the proof of Theorem \ref{thmBK}.

\section{Proof of Theorem \ref{lemBS2}}\label{secces}
To establish an optimal bound for  $\mathfrak{S}(\bm{b},\bm{h},1;p^2)$, we require point counting on
a (generically zero-dimensional) variety in $\bbm{F}_p^{11}$.
For given $\bm{b}\in \bbm{F}_p^4$ and $\bm{h}\in\bbm{F}_p^2$, we define $\mV_{11}$ as the variety of $11$-tuples
\begin{equation}
\label{11tuple}
\bm{v}=(u,x_1,x_2,x_{11},\cdots, x_{14},x_{21},\cdots,x_{24})\in\bbm{F}_p^{11}
\end{equation}
satisfying the following eleven equations
\begin{subequations}\begin{align}
&x_i x_{ij}=(u+b_j)\ol{x}_{ij},\ \text{for}\ 1\le i\le 2, 1\le j\le4;\label{eqce1}\\
&\sum_{j=1}^4x_{ij}+h_i=0,\ \text{for}\ 1\le i\le 2;\label{eqce2}\\
&\sum_{i=1}^2\sum_{j=1}^4\ol{x}_{ij}=0.\label{eqce3}
\end{align}
\end{subequations}
Let $\mathcal{K}(\bm{b},\bm{h};p)$ denote the number of $\bm{v}\in\mV_{11}$ with
\begin{equation}
\label{additional-conditions}
x_1\neq x_2, \quad x_i,\ x_{ij},\ u+b_j\in \mathbb{F}_p^{\times}
\end{equation}
for $1\le i\le 2$, $1\le j\le 4$.

Since $\mV_{11}$ is defined by eleven polynomials in $\bbm{F}_p^{11}$, it generically constitutes a zero-dimensional variety. Thus, we expect $\mathcal{K}(\bm{b},\bm{h};p)\ll 1$ to hold except in degenerate cases. We make this precise in the following key lemma.

\begin{lemma}\label{thmK}
For $\bm{b}\in\bbm{F}_p^4$ and $\bm{h}\in\bbm{F}_p^2$ with $\bm{b}\notin\mV^\Delta$, we have
\begin{align}
\mathcal{K}(\bm{b},\bm{h};p)\ll
\begin{cases}
p & \text{if}\ \bm{h}=(0,0),\\
p & \text{if}\  (\bm{b},\bm{h})\in \mV_4^{bad}\times\mV_2^{bad}(\bm{b}),\\
1 & \text{otherwise}.
\end{cases}
\end{align}
Here $\mV_4^{bad}\subset\mathbb{F}_p^4$ is a variety defined by a homogeneous polynomial of bounded degree, while for any given $\bm{b}$, the variety $\mV_2^{bad}(\bm{b})\subset\mathbb{F}_p^2$ is defined by
$O(1)$ linear homogeneous polynomials.
\end{lemma}

We now proceed to establish Theorem \ref{lemBS2} conditional on Lemma \ref{thmK}, deferring the proof of Lemma \ref{thmK} to the subsequent section.
\begin{proof}[Proof of Theorem \ref{lemBS2}]
Note that $\K\left(s_i(r+b_j);p^2\right)$ vanishes unless $(s_i(r+b_j),p)=1$, a convention maintained through this proof.
Recall the representation
\[
\K\left(s_i(r+b_j);p^2\right)=\frac 1p\ssum_{x\ppmod {p^2}}e_{p^2}\l(s_ix+(r+b_j)\ol{x}\r).
\]
Through the parameterization
\[
s_i=x_i(1+y_ip),\quad r=u+vp,\quad x=x_{ij}(1+y_{ij}p),
\]
where $x_{ij}, y_{ij}, x_i, y_j, u, v$ are defined modulo $p$ with $(x_i(u+b_j)x_{ij},p)=1$, we have
\begin{align*}
\K\left(s_i(r+b_j);p^2\right)=&\frac1p\ssum_{x_{ij}\ppmod p}e_{p^2}\l(x_ix_{ij}+(u+b_j)\ol{x}_{ij}\r)e_p\l(x_iy_ix_{ij}+v\ol{x}_{ij}\r)\\
&\times\sum_{y_{ij}\ppmod p}e_p\l(y_{ij}(x_ix_{ij}-(u+b_j)\ol{x}_{ij})\r).
\end{align*}
The $y_{ij}$ summation vanishes unless $p\mid (x_ix_{ij}-(u+b_j)\ol{x}_{ij})$, reducing to $p$ when nonvanishing. Consequently,
\begin{align*}
\K\left(s_i(r+b_j);p^2\right)=\ssum_{\substack{x_{ij}\ppmod p\\ x_ix_{ij}\equiv (u+b_j)\ol{x}_{ij}\ppmod p}}e_{p^2}\l(x_ix_{ij}+(u+b_j)\ol{x}_{ij}\r)e_p\l(x_iy_ix_{ij}+v\ol{x}_{ij}\r).
\end{align*}
Substituting this and the identity
\[
e_{p^2}(h_is_i)=e_{p^2}(h_ix_i)e_p(h_ix_iy_i)
\]
into $\mathfrak{S}(\bm{b},\bm{h},1;p^2)$ shows
\begin{equation}\label{eqS211}
\begin{aligned}
\mathfrak{S}(\bm{b},\bm{h},1;p^2)=&\sum_{u\ppmod p}\mathop{\ssum\cdots\ssum}_{\substack{x_i, x_{ij}, (i=1,2; j=1,2,3,4)\\ x_i x_{ij}\equiv (u+b_j)\ol{x}_{ij}\ppmod p\\ x_1\neq x_2}}e_{p^2}\B(\sum_{i=1}^2\B(h_ix_i+\sum_{j=1}^4 \l(x_i x_{ij}+(u+b_j)\ol{x}_{ij}\r)\B)\B)\\
&\mathop{\sum\sum}_{y_1,y_2\ppmod p}e_p\B(\sum_{i=1}^2x_iy_i\B(\sum_{j=1}^4x_{ij}+h_i\B)\B)\sum_{v\ppmod p}e_p\B(v\sum_{i=1}^2\sum_{j=1}^4\ol{x}_{ij}\B),
\end{aligned}
\end{equation}
where the product of the sums over $y_i$ and $v$ vanish unless
\[
\sum_{j=1}^4x_{ij}+h_i\equiv0\pmod p \quad \text{and}\quad  \sum_{i=1}^2\sum_{j=1}^4\ol{x}_{ij}\equiv0 \pmod p.
\]
Applying this to \eqref{eqS211} yields
\begin{align*}
\mathfrak{S}(\bm{b},\bm{h},1;p^2)\ll p^3\mathcal{K}(\bm{b},\bm{h};p),
\end{align*}
which, combined with Lemma \ref{thmK}, establishes Theorem \ref{lemBS2}.
\end{proof}

\section{Counting rational points on the variety }
\label{variety-sec}
This section is devoted to the proof of Lemma \ref{thmK}. Our strategy is to reduce, in several steps, the problem of counting the vectors $\bm{v}$ on the variety $\mathcal{V}_{11}$ (subject to the additional conditions \eqref{additional-conditions}), first to the problem of counting the corresponding triples $(u,x_{11},x_{21})$, and then just to counting the corresponding values of $u$.
We analyze the relations for $u$ through polynomial substitution, which generically admits
$O(1)$ solutions. Degeneracy occurs precisely when both $\bm{b}$ and $\bm{h}$ lie on the varieties defined by homogeneous polynomials of bounded degree as indicated in the statement of Lemma~\ref{thmK}.

We remark that the statement of Lemma~\ref{thmK} is trivially true for primes $p$ of bounded size, and that therefore, throughout this section, we may (and will) assume that $p$ is larger than a certain absolute lower bound, so that various absolute positive integer constants appearing in our arguments will be units modulo $p$.

\subsection{Reduction to solutions of polynomials}
From \eqref{eqce1}, we establish the relations
\begin{align}\label{eqxx}
\frac{x_{ij}}{x_{i1}}=\pm\l(\frac{u+b_j}{u+b_1}\r)^{\frac12} \quad\text{and}\quad \frac{x_{2j}}{x_{1j}}=\pm \l(\frac{x_1}{x_2}\r)^{\frac12}
\end{align}
for $1\le i\le 2, 1\le j\le4$, where we fix an arbitrary branch of the square root from $\mathbb{F}_p^{\times 2}\to\mathbb{F}_p^{\times}$, and merely using the square root notation indicates that the argument lies in $\mathbb{F}_p^{\times 2}$ (and thus the square root is defined) without additional notice. Substituting these relations into \eqref{eqce2} and \eqref{eqce3} produces the system
\begin{subequations}\begin{align}
&x_{11}f_1(u)+h_1=0;\label{eqce2'}\\
&x_{21}f_2(u)+h_2=0;\label{eqce2''}\\
&\ol{x}_{11}g_1(u)+\ol{x}_{21}g_2(u)=0,\label{eqce3'}
\end{align}
\end{subequations}
where the functions are defined as
\begin{align}\label{deffi}
f_i(u)=Q_i\B(1,\l(\frac{u+b_2}{u+b_1}\r)^{\frac12},\l(\frac{u+b_3}{u+b_1}\r)^{\frac12},\l(\frac{u+b_4}{u+b_1}\r)^{\frac12}\B),
\end{align}
\begin{align}\label{defgi}
g_i(u)=Q_i\B(1,\l(\frac{u+b_1}{u+b_2}\r)^{\frac12},\l(\frac{u+b_1}{u+b_3}\r)^{\frac12},\l(\frac{u+b_1}{u+b_4}\r)^{\frac12}\B)
\end{align}
with
$Q_i(x,y,z,w)\in \{x\pm y\pm z\pm w\}$.

Let $\mathcal{K}_1(\bm{b},\bm{h};p)$ denote the number of triples $(u,x_{11},x_{21})\in\mathbb{F}_p^3$ satisfying one of the $O(1)$ instances of the system \eqref{eqce2'}--\eqref{eqce3'},
with $f_i$, $g_i$ as in \eqref{deffi}--\eqref{defgi}. We have proved that, for every tuple $\bm{v}\in\mathcal{V}_{11}$ counted by $\mathcal{K}(\bm{b},\bm{h};p)$ and labeled as in \eqref{11tuple}, the triple $(u,x_{11},x_{21})$ is among the triples counted by $\mathcal{K}_1(\bm{b},\bm{h};p)$. On the other hand, the
relations \eqref{eqce1} and \eqref{eqxx} demonstrate that, for given parameters $\bm{h}, \bm{b}$, all remaining variables in $\bm{v}$ are determined by the triple $(u, x_{11}, x_{21})$ up to $O(1)$ choices.
From this it follows that
\[ \mathcal{K}(\bm{b},\bm{h};p)\asymp\mathcal{K}_1(\bm{b},\bm{h};p), \]
and it suffices to bound the latter count.

Moreover, for a given $u$, the value of $x_{11}$ and $x_{21}$ can be resolved from \eqref{eqce2'}--\eqref{eqce3'}, except in the degenerate case where at least two of $f_1(u), f_2(u)$,
$g_1(u)$, and $g_2(u)$ vanish simultaneously.
The following subsection \S\ref{counting-solutions-subsec} is devoted to determining $u$ for which this happens,
with analysis proceeding
in two cases. When either $h_1=0$ or $h_2=0$,
$u$ will be determined to within $O(1)$ choices by the corresponding $f_i(u)=0$. For non-zero $h_1$
and $h_2$, substituting \eqref{eqce2'} and \eqref{eqce2''} into \eqref{eqce3'} yields
\[
\ol{h}_1f_1(u)g_1(u)+\ol{h}_2f_2(u)g_2(u)=0,
\]
which will similarly determine $u$ up to $O(1)$ choices outside degenerate cases. We put everything together and prove Lemma~\ref{thmK} in \S\ref{proof-lemma-subsec}.

\subsection{Lemmas for counting solutions of polynomials}
\label{counting-solutions-subsec}
\begin{lemma}\label{lempower}
For any rational function $Q(x,y,z,w)$ in four variables, the product
\[
F(x,y,z,w)=\prod Q(\pm x^{\frac12},\pm y^{\frac12},\pm z^{\frac12},\pm w^{\frac12})
\]
taken over all sign combinations remains a rational function of $x,y,z,w$.
\end{lemma}
\begin{proof}
We first establish the polynomial case. Consider the auxiliary polynomial
\[
P(x,y,z,w)=\prod Q(\pm x,\pm y,\pm z,\pm w).
\]
The key observation is the invariance property
\begin{align*}
P(\epsilon_1 x,\epsilon_2 y,\epsilon_3 z,\epsilon_4 w)=P(x,y,z,w)\quad \text{for any}\ \epsilon_i=\pm1.
\end{align*}
This symmetry implies that all monomials in $P(x,y,z,w)$ must contain even powers of each variable; indeed, this is clear from $P(x,y,z,w)=(1/16)\sum P(\pm x,\pm y,\pm z,\pm w)$. Therefore, the substitution $x\rightarrow x^{\frac12}$, $y\rightarrow y^{\frac12}$, $z\rightarrow z^{\frac12}$, $w\rightarrow w^{\frac12}$ yields a well-defined polynomial
\[
F(x,y,z,w)=P(x^{\frac12}, y^{\frac12}, z^{\frac12}, w^{\frac12}).
\]
The rational case follows by considering numerator and denominator polynomial separately.
\end{proof}

\begin{lemma}\label{lemsolution1}
For any given $\bm{b}\notin \mV^\Delta$, both
each one of the congruence equations
\begin{align}\label{eqQ1}
f_i(u)&\equiv0\pmod p,\\
\label{eqQ1'}
g_i(u)&\equiv0 \pmod p
\end{align}
possess $O(1)$ solutions.
\end{lemma}
\begin{proof}
Applying Lemma \ref{lempower}, we eliminate all half-powers through the symmetric polynomial construction
\begin{align*}
F(x,y,z,w)&=\prod Q_i(\pm x^{\frac12},\pm y^{\frac12},\pm z^{\frac12},\pm w^{\frac12})\\
&=64xyzw-\l((x+y-z-w)^2-4(xy+zw)\r)^2.
\end{align*}
For equation \eqref{eqQ1}, solutions must satisfy
\[
F(u+b_1,u+b_2,u+b_3,u+b_4)=0,
\]
which, by a straightforward calculation, simplifies to the linear equation\footnote{This equation as well as other equations in the next lemma can be calculated by math softwares such as \emph{Mathematica} easily.}
\[
c_1u+c_0=0
\]
 with the coefficient
\[
c_1=8(b_1+b_2-b_3-b_4)(b_1-b_2+b_3-b_4)(b_1-b_2-b_3+b_4)
\]
and $c_0$ a certain explicit degree 6 form in $b_i$. We claim that, for $\bm{b}\not\in\mathcal{V}^{\Delta}$, we cannot have $c_1=0$ and $c_0=0$ simultaneously.
The symmetry of $F(u+b_1,u+b_2,u+b_3,u+b_4)$ in $b_i$ produces the symmetry for the coefficients $c_0$ and $c_1$.
Without loss of generality, we apply $b_1+b_2-b_3-b_4=0$ in place of $c_1=0$ to the expression of $c_0$, and it follows that
\[
c_0=-16(b_1-b_3)^2(b_2-b_3)^2\neq0
\]
unless $\bm{b}\in\mV^\Delta$. Thus the two coefficients cannot vanish simultaneously, yielding at most one solution for $u$.

For \eqref{eqQ1'}, the transformed equation
\begin{align}\label{eqF-}
F\l((u+b_1)^{-1},(u+b_2)^{-1},(u+b_3)^{-1},(u+b_4)^{-1}\r)=0
 \end{align}
reduces to a ninth-degree equation
\[
c_9u^9+c_8u^8+\cdots=0,
\]
where the leading coefficient
\[
c_9=-8(b_1+b_2-b_3-b_4)(b_1-b_2+b_3-b_4)(b_1-b_2-b_3+b_4).
\]
Applying $b_1+b_2-b_3-b_4=0$ to the expression of $c_8$ gives
\[
c_8=48(b_1-b_3)^2(b_2-b_3)^2\neq0
\]
for $\bm{b}\notin\mV^\Delta$.
The fact that $c_9$ and $c_8$ cannot vanish simultaneously yields at most nine solutions for $u$.

\end{proof}

\begin{lemma}\label{lemsolution2}
Given the function definitions $f_i$, $g_i$ from \eqref{deffi} and \eqref{defgi} with distinct $Q_1\neq Q_2$, for $\bm{b}\notin\mV^\Delta$, the equation
\begin{align}\label{eqQ11}
\ol{h}_1f_1(u)g_1(u)+\ol{h}_2f_2(u)g_2(u)=0
\end{align}
admits $O(1)$ solutions except when $(\bm{b},\bm{h})\in \mV_4^{bad}\times\mV_2^{bad}(\bm{b})$. Here $\mV_4^{bad}\subset\mathbb{F}_p^4$ denotes an algebraic variety determined by a homogeneous polynomial of bounded degree, and $\mV_2^{bad}(\bm{b})\subset\mathbb{F}_p^2$ represents a variety specified by  $O(1)$ linear homogeneous polynomials for given $\bm{b}$.
\end{lemma}

\begin{proof}
Building upon Lemma \ref{lemsolution1}, we focus our analysis on the non-degenerate case where $f_1(u)f_2(u)g_1(u)g_2(u)\not=0$. To eliminate half-powers terms, we apply Lemma \ref{lemsolution2} by constructing the symmetric polynomial
\[
F(x,y,z,w)=\prod Q(\pm x^{\frac12},\pm y^{\frac12},\pm z^{\frac12},\pm w^{\frac12}),
\]
where the rational function
\[
Q(x,y,z,w)=\ol{h}_1Q_1(x,y,z,w)Q_1\l(x^{-1},y^{-1},z^{-1},w^{-1}\r)+\ol{h}_2Q_2(x,y,z,w)Q_2\l(x^{-1},y^{-1},z^{-1},w^{-1}\r).
\]
Consequently, all solutions to equation \eqref{eqQ11} must satisfy the transformed equation
\begin{align}\label{eqeqF}
F(u+b_1,u+b_2,u+b_3,u+b_4)=0.
\end{align}

When ignoring variables orders, the expression for $Q$ (and consequently $F$) is determined exclusively by the parity of sign differences between $Q_1, Q_2\in\{x\pm y\pm z\pm w\}$. This leads us to partition equation \eqref{eqeqF} into two cases according to this sign difference count.

\emph{Case I:  two different signs.} Without loss of generality, we consider the specific configuration $Q_1(x,y,z,w)=x+y+z+w$ and $Q_2(x,y,z,w)=x+y-z-w$ in this case.
A direct calculation yields
\begin{align*}
Q=(\ol{h}_1+\ol{h}_2)\l(4+\frac{x}{y}+\frac{y}{x}+\frac{z}{w}+\frac{w}{z}\r) +(\ol{h}_1-\ol{h}_2)\l(\frac{x}{z}+\frac{z}{x}+\frac{x}{w}+\frac{w}{x}+\frac{y}{z}+\frac{z}{y}+\frac{y}{w}+\frac{w}{y}\r),
\end{align*}
where $\ol{h}_1+\ol{h}_2$ and $\ol{h}_1-\ol{h}_2$ are not simultaneously zero.

When $\ol{h}_1+\ol{h}_2=0$, expansion of \eqref{eqeqF} yields a $24$th-degree polynomial
\[
c_{24}u^{24}+c_{23}u^{23}+\cdots= 0
\]
with the leading coefficient
\[
c_{24}=1048576\l((b_1-b_2)^2+(b_3-b_4)^2\r)^4.
\]
Thus, $u$ admits at most $24$ distinct solutions unless $\bm{b}$ satisfies $(b_1-b_2)^2+(b_3-b_4)^2=0$.

For $\ol{h}_1-\ol{h}_2=0$, simplification of \eqref{eqeqF} gives an $8$th-degree polynomial
\[
c_{8}u^{8}+c_{7}u^{7}\cdots=0,
\]
where
\[
c_8=(b_1-b_2)^8(b_1+b_2-b_3-b_4)^8(b_3-b_4)^8.
\]
Hence, $u$ has at most $8$ solutions unless $(b_1-b_2)(b_1+b_2-b_3-b_4)(b_3-b_4)=0$.

When $(\ol{h}_1+\ol{h}_2)(\ol{h}_1-\ol{h}_2)\not=0$, we define
\begin{align}\label{eqdeft}
t=\frac{\ol{h}_1+\ol{h}_2}{\ol{h}_1-\ol{h}_2},
\end{align}
where $t\notin \{0,\pm 1\}$. Treating $t$ and $\bm{b}$ as parameters, a direct calculation shows that \eqref{eqeqF} reduces to a $24$th-degree polynomial equation
\[
c_{24}u^{24}+c_{23}u^{23}+c_{22}u^{22}+\cdots=0
\]
with the leading coefficient
\[
c_{24}=1048576t^8(t^2-1)^2\l(\l((b_1-b_2)^2+(b_3-b_4)^2\r)^2t^2-4(b_1-b_2)^2(b_3-b_4)^2\r)^2.
\]
For $\bm{b}\notin\mV^\Delta$ and $t\not= 0,\pm 1$, not both the $t^2$ and constant term coefficients in the final factor above can vanish simultaneously, and so the condition $c_{24}=0$ holds precisely when $\bm{h}$ satisfies
\begin{align}\label{eqtvalue}
t=\frac{\ol{h}_1+\ol{h}_2}{\ol{h}_1-\ol{h}_2}=\pm \frac{2(b_1-b_2)(b_3-b_4)}{(b_1-b_2)^2+(b_3-b_4)^2}.
\end{align}
Substituting \eqref{eqtvalue} into subsequent coefficients yields $c_{23}=0$ and
\begin{align*}
c_{22}=&4294967296\frac{(b_1-b_2)^{12}(b_3-b_4)^{12}}{(b_1-b_2)^2+(b_3-b_4)^{2}}\\
&\times (b_1+b_2-b_3-b_4)^2(b_1-b_2+b_3-b_4)^6(b_1-b_2-b_3+b_4)^6.
\end{align*}
Consequently, both $c_{24}$ and $c_{22}$ vanish simultaneously if and only if
\[
(b_1+b_2-b_3-b_4)(b_1-b_2+b_3-b_4)(b_1-b_2-b_3+b_4)=0
\]
and $\bm{h}$ satisfies \eqref{eqtvalue}. Otherwise, the variable $u$ admits at most
$24$ distinct values.

\emph{Case II:  one different sign.}
In this case, we take $Q_1(x,y,z,w)=x+y+z+w$ and $Q_2(x,y,z,w)=x+y+z-w$ for example. A direct calculation shows
\begin{align*}
Q=(\ol{h}_1+\ol{h}_2)\l(4+\frac{x}{y}+\frac{y}{x}+\frac{x}{z}+\frac{z}{x}+\frac{y}{z}+\frac{z}{y}\r) +(\ol{h}_1-\ol{h}_2)\l(\frac{x}{w}+\frac{w}{x}+\frac{y}{w}+\frac{w}{y}+\frac{z}{w}+\frac{w}{z}\r),
\end{align*}
where $\ol{h}_1+\ol{h}_2$ and $\ol{h}_1-\ol{h}_2$ also cannot vanish simultaneously.

After substituting this into \eqref{eqeqF}, a simplification yields
\[
40960000u^{32}+c_{31}u^{31}+\cdots=0
\]
when $\ol{h}_1+\ol{h}_2=0$, and it reduces to
\[
5308416u^{32}+c_{31}u^{31}+\cdots=0
\]
when $\ol{h}_1-\ol{h}_2=0$.
Consequently, we establish that $u$ admits at most $32$ possible values when $(\ol{h}_1+\ol{h}_2)(\ol{h}_1-\ol{h}_2)=0$.

Given $(\ol{h}_1+\ol{h}_2)(\ol{h}_1-\ol{h}_2)\not=0$, we employ the parameter $t$ as defined in \eqref{eqdeft}. A direct calculation then simplifies \eqref{eqeqF} to a $32$nd-degree polynomial equation
\[
c_{32}u^{32}+c_{31}u^{31}+\cdots=0
\]
where the leading coefficient satisfies
\[
c_{32}=65536(t^2-1)^6(25t^2-9)^2\neq 0
\]
except when $t=\pm 3/5$. Substituting $t=\pm 3/5$ into subsequent coefficients yields $c_{32}=\cdots=c_{29}=0$, with
\[
c_{28}=c\l(b_1^2-6b_1b_2+b_2^2-6b_1b_3-6b_2b_3+b_3^2+10b_1b_4+10b_2b_4+10b_3b_4-15b_4^2\r)^2,
\]
where $c=618475290624/6103515625$. Consequently, $u$ admits at most $32$ solutions except when $\bm{h}$ and $\bm{b}$ simultaneously satisfy
\[
5(\ol{h}_1+\ol{h}_2)=\pm 3(\ol{h}_1-\ol{h}_2)
\]
and
\[
b_1^2-6b_1b_2+b_2^2-6b_1b_3-6b_2b_3+b_3^2+10b_1b_4+10b_2b_4+10b_3b_4-15b_4^2=0.
\]
We thereby establish the lemma.
\end{proof}

\begin{lemma}\label{lemb1234}
For any vector $\bm{b}=(b_1,b_2,b_3,b_4)$ satisfying the congruence relations
\begin{align}
b_1+b_2\equiv b_3+b_4\pmod p,\label{eqb12341}\\
\ol{b}_1+\ol{b}_2\equiv \ol{b}_3+\ol{b}_4\pmod p,\label{eqb12342}
\end{align}
the squared tuple $\l(b_1^2, b_2^2, b_3^2, b_4^2\r)$ lies in the variety $\mV^\Delta$ modulo $p$.
\end{lemma}
\begin{proof}
When $b_1+b_2\equiv b_3+b_4\equiv0\pmod p$, the conclusion follows immediately. Otherwise, we deduce from \eqref{eqb12342} the relation
\[
\ol{b}_1\ol{b}_2(b_1+b_2)\equiv \ol{b}_1+\ol{b}_2\equiv \ol{b}_3+\ol{b}_4\equiv\ol{b}_3\ol{b}_4(b_3+b_4)\pmod p,
\]
implying $b_1b_2\equiv b_3b_4 \ (\bmod \ p)$. Combining with \eqref{eqb12341}, this establishes
\[
(b_1-b_2)^2\equiv(b_3-b_4)^2\pmod p.
\]
Consequently, either
\[
b_1-b_2+b_3-b_4\equiv 0\pmod p,
\]
or
\[
b_1-b_2-b_3+b_4\equiv 0\pmod p
\]
must be satisfied. The lemma follows by combining these with \eqref{eqb12341}.
\end{proof}

\subsection{Proof of Lemma \ref{thmK}}
\label{proof-lemma-subsec}
Lemma \ref{lemb1234} along with a moment's reflection shows that the system
\begin{align*}
\begin{cases}
f_i(u)=0,&\\
g_i(u)=0&
\end{cases}
\end{align*}
has no solution when $\bm{b}\notin\mV^\Delta$, implying $f_i(u)$ and $g_i(u)$ cannot both vanish in this case. We analyze $\mathcal{K}_1(\bm{b},\bm{h};p)$ by examining four cases classified according to the values of $h_i$.

When $h_1= h_2=0$, we obtain $f_1(u)=f_2(u)=0$ and thus $g_1(u)g_2(u)\not= 0$. By Lemma \ref{lemsolution1} and equation \eqref{eqce3'}, the variable $u$ can take at most
$O(1)$ values, while $x_{11}$, $x_{21}$ are uniquely determined by each other; thus, in this case,
\[
\mathcal{K}_1(\bm{b},\bm{h};p)\ll p.
\]

When $h_1=0\not=h_2$, we have $f_1(u)= 0 $ with $f_2(u)g_1(u)\not= 0$. Lemma \ref{lemsolution1} shows that $u$ can take at most $O(1)$ values, while $x_{11}$, $x_{21}$ are uniquely determined by equations \eqref{eqce2''} and \eqref{eqce3'}. Consequently,
\[
\mathcal{K}_1(\bm{b},\bm{h};p)\ll 1.
\]
The symmetric case $h_2=0\not=h_1$ yields an identical bound.

When $h_1h_2\not= 0$ with $g_1(u)g_2(u)=0$, the non-degeneracy condition $f_1(u)f_2(u)\not=0$ is automatically satisfied. Lemma \ref{lemsolution1} restricts $u$ to
$O(1)$ possibilities, while equations \eqref{eqce2'} and \eqref{eqce2''} uniquely determine $x_{11}$ and $x_{21}$ . This yields the bound
\[
\mathcal{K}_1(\bm{b},\bm{h};p)\ll 1.
\]

Given $h_1h_2\not=0$ and $g_1(u)g_2(u)\not=0$, the variables $x_{11}$ and $x_{21}$ are uniquely determined by equations \eqref{eqce2'} and \eqref{eqce2''} for fixed $u$, $\bm{h}$ and $\bm{b}$.
To determine $u$, we substitute \eqref{eqce2'} and \eqref{eqce2''} into \eqref{eqce3'}, yielding
\[
\ol{h}_1f_1(u)g_1(u)+\ol{h}_2f_2(u)g_2(u)=0.
\]
For $x_1\neq x_2$ with $g_1(u)g_2(u)\not=0$, equations \eqref{eqxx} and \eqref{eqce3'} ensure $g_1(u)\neq g_2(u)$, which implies $Q_1\neq Q_2$.
By Lemma \ref{lemsolution2}, $u$ admits at most
$O(1)$ values unless $(\bm{b},\bm{h})\in \mV_4^{bad}\times\mV_2^{bad}(\bm{b})$. Consequently, we obtain the bound
\begin{align*}
\mathcal{K}_1(\bm{b},\bm{h};p)\ll
\begin{cases}
p, & \text{if}\  (\bm{b},\bm{h})\in \mV_4^{bad}\times\mV_2^{bad}(\bm{b});\\
1, & \text{otherwise}.
\end{cases}
\end{align*}
Keeping in mind that $\mathcal{K}(\bm{b},\bm{h};p)\asymp\mathcal{K}_1(\bm{b},\bm{h};p)$, this completes the proof via exhaustive case analysis.

\section{Evaluation of the moment}
This section presents the asymptotic evaluation of moments for twisted $L$-function as stated in Theorem \ref{thmL}. We begin by performing several reduction steps and auxiliary estimates in \S\ref{reduction-subsec}, and we combine everything into a proof of Theorem~\ref{thmL} in \S\ref{thmLproof-subsec}.

\subsection{Reduction steps}
\label{reduction-subsec}
Following established methodology (see \cite[Section 6.1]{BFK+17a}), we outline the key steps below.

Using the approximate functional equation from (\cite[Theorem 5.3]{IK04}), we express the central values as a convergent series
\begin{align*}
L(1/2,f_1\otimes\chi)\ol{L(1/2,f_2\otimes\chi)}=&\mathop{\sum\sum}_{m,n\ge1}\frac{\lambda_1(m)\lambda_2(n)\chi(m)\ol{\chi}(n)}{(mn)^{\frac12}}V\l(\frac{mn}{q^2}\r)\\
&+\ve(f_1,f_2,\chi)\mathop{\sum\sum}_{m,n\ge1}\frac{\lambda_1(m)\lambda_2(n)\ol{\chi}(m)\chi(n)}{(mn)^{\frac12}}V\l(\frac{mn}{q^2}\r),
\end{align*}
where $\ve(f_1,f_2,\chi)=\pm 1$ denotes the root number, and $V(x)$ is a smooth function with rapid decay for $x\gg q^\ve$. Character orthogonality yields that the average over primitive characters reduces to combinations of
\[
\frac{1}{\vp^{*}(q)}\sum_{d\mid q}\mu\l(\frac qd\r)\frac{\vp(d)}{(MN)^{\frac12}}\sum_{\substack{m\equiv \pm n\ppmod d\\ (mn,q)=1}}\lambda_1(m)\lambda_2(n)V\l(\frac{mn}{q^2}\r).
\]

The main term in \eqref{eqmoment} originates from the diagonal term $m=n$, computable via Mellin inversion and contour shifting (see \cite[Section 3.1]{BM2015}).
This calculation requires no assumptions (such as factorability of $q$ or bounds toward the Ramanujan--Petersson conjecture) beyond those already stated in Theorem~\ref{thmL}.
The off-diagonal contribution requires bounding
\begin{equation}
\label{mBpm}
\mB^{\pm}(M,N):=\frac{1}{\vp^{*}(q)}\sum_{d\mid q}\mu\l(\frac qd\r)\frac{\vp(d)}{(MN)^{\frac12}}\sum_{\substack{m\equiv \pm n\ppmod d\\ (mn,q)=1\\ m\neq n}}\lambda_1(m)\lambda_2(n)W_1\l(\frac mM\r)W_2\l(\frac nN\r)
\end{equation}
for $N\ge M$ and $MN\le q^{2+\ve}$, where $W_1, W_2$ are test functions supported on $[1,2]$ with derivatives satisfying
$W_1^{(j)}, W_2^{(j)}\ll_j q^\ve$.

To begin with, we establish two crucial estimates for $\mB^{\pm}(M,N)$. Some of our estimates will be in terms of an admissible exponent $\theta\geqslant 0$ toward the Ramanujan--Petersson conjecture for Hecke--Maa\ss{} newforms $f$ of arbitrary level $\ell$, that is, an exponent such that the corresponding Hecke eigenvalues satisfy
\begin{equation}
\label{toward-RPC}
\lambda_f(n)\ll n^{\theta+\ve};
\end{equation}
currently, $\theta=\frac{7}{64}$ is known to be admissible by the work of Kim--Sarnak~\cite{Kim03}. Of course, the corresponding bound $O_{\ve}(n^{\ve})$ is trivial for Eisenstein series and known for holomorphic newforms by the work of Deligne~\cite{Del74}.

First, we have the trivial bound
\begin{align}\label{eqtbSMN}
\mB^{\pm}(M,N)\ll N^\theta\frac{(MN)^{\frac12}}{q}(MNq)^\ve
\end{align}
for $N\ge M\ge1$, derived via the Cauchy--Schwarz inequality using
\[
\lambda_2(n)\ll n^{\theta+\ve},\quad \sum_{m\le M}|\lambda_1(m)|^2\ll M.
\]

The second estimate applies specifically to balanced cases where $M$ and $N$ are comparable in size.
\begin{lemma}
\label{lemmaBmn}
For $N\ge M\ge1$ and $MN\le q^{2+\ve}$, we have
\begin{align}\label{eqbalancedS}
\mB^{\pm}(M,N)\ll \l(\frac{N}{M}\r)^{\frac14}q^{-\frac14+\ve}+\l(\frac{N}{M}\r)^{\frac12}q^{-\frac12+\ve}+q^{-\frac12+2\theta+\ve}.
\end{align}
\end{lemma}

\begin{proof}
We will prove Lemma~\ref{lemmaBmn} by refining the arguments of \cite{BM2015} to remove all dependence on the Ramanujan--Petersson conjecture. For ease of reference, we recall the notations from \cite[(3.4)--(3.8)]{BM2015}:
\begin{align*}
S_{N,M,d,q}&=\frac{d}{(MN)^{1/2}}\sum_{\substack{m\equiv n\ppmod d\\(mn,q)=1\\m\neq n}}\lambda_1(m)\lambda_2(n)W_1\l(\frac{m}M\r)W_2\l(\frac{n}N\r),\\
\mathcal{D}(\ell_1,\ell_2,h,N,M)&=\sum_{\ell_1n-\ell_2m=h}\lambda_1(m)\lambda_2(n)W_1\l(\frac{\ell_2m}M\r)W_2\l(\frac{\ell_1n}N\r),\\
\mathcal{S}(\ell_1,\ell_2,d,N,M)&=\sum_r\mathcal{D}(\ell_1,\ell_2,rd,N,M).
\end{align*}
Bounds on $S_{N,M,d,q}$ directly lead to corresponding estimates on $\mathcal{B}^{\pm}(M,N)$. In particular, the estimate \eqref{eqbalancedS} follows from the combination of the following two bounds:
\begin{alignat}{3}
\label{eq3.11'}
S_{N, M, d, q}&\ll (qN)^{\varepsilon}  \B( \frac{(Nq)^{\frac12}}{M^{\frac12}} + \frac{N^{\frac34}  }{ M^{\frac14}} +  \frac{N^{\frac14} q^{\frac34}}{M^{\frac14} } + N^{\frac12}q^{\frac14} \B)&\qquad&(N\geqslant 20M),\\
\label{eq3.12'}
S_{N, M, d, q}&\ll \frac{q^{\theta+\ve}N^{1+\theta}}{M^{\frac12}}&&(N\geqslant M).
\end{alignat}
These bounds were originally formulated and proved in \cite[(3.11), (3.12)]{BM2015} (the latter without the $q^{\theta}$ factor) under the Deligne's bound for $\lambda_j(n)$, $j=1,2$. We now prove \eqref{eq3.11'} and \eqref{eq3.12'} independently of the Ramanujan--Petersson conjecture.

The sum $S_{N,M,d,q}$ is related to the other two sums above (which are in turn treated by the shifted convolution sum methods) in \cite[(3.10)]{BM2015}. Using \eqref{toward-RPC} in place of Deligne's bound in \cite[(3.10)]{BM2015} introduces an additional $(f_1f_2)^{\theta+\varepsilon}$ factor in the final expression:
\begin{equation}
\label{4310-theta}
S_{N, M, d, q}\ll \frac{d }{(NM)^{\frac12}} \sum_{\substack{g_1\mid f_1 \mid q\\ g_2\mid f_2 \mid q}}   (f_1f_2)^{\theta + \varepsilon}\Bigl|   \mathcal{S}(f_1g_1, f_2g_2, d, N, M)\Bigr|. 
\end{equation}

We begin by proving \eqref{eq3.11'}. Crucial in deriving its precedent, \cite[(3.11)]{BM2015}, is the upper bound on $\mathcal{S}(\ell_1,\ell_2,d,N,M)$ for $N\geq 20M$ in \cite[Proposition 8]{BM2015}, whose proof we note is entirely independent of the Ramanujan--Petersson conjecture. Indeed, up to the intermediate bound \cite[(8.17)]{BM2015},
\[ \mathcal{S}(\ell_1,\ell_2,d,N,M)\ll (dN)^{\varepsilon}  (\ell_1\ell_2,d)^{1/2} \B(\frac{N}{d^{\frac12}} +\frac{N^{\frac54}M^{\frac14} }{d} + \frac{N^{\frac34}M^{\frac14} }{d^{\frac14}} + \frac{NM^{\frac12} }{d^{\frac34}}\B), \]
the proof proceeds by shifted convolution sum methods and the Hecke eigenvalues $\lambda_j(n)$ are only estimated on average using the bounds of Rankin--Selberg and Wilton. To obtain \cite[Proposition 8]{BM2015}, the extraneous factor $(\ell_1\ell_2,d)^{1/2}$ must be removed. This is accomplished in the argument below \cite[(8.17)]{BM2015}, which uses the same decompositions as our \eqref{all-common-factors} and \eqref{after-Hecke} below, where $(\ell_1'\ell_2',d)=1$, whence $(\ell_1'g\ell_2'h,d')=(gh,d')\mid gh$, and using only the trivial bounds $\lambda_j(n)\ll n^{1/2}$ suffices to complete the proof of \cite[Proposition 8]{BM2015} because the combined contribution of each term $N^uM^v/d^w$ in the above bound is seen to be
\begin{equation}
\label{a-bit-more-care}
\ll (gh)^{1/2}\Big(\frac{\delta_1\delta_2}{gh}\Big)^{\frac12}\frac{(N/\delta\delta_1\delta_2\tilde{\ell})^u(M/\delta\delta_1\delta_2\tilde{\ell})^v}{(d/\delta\delta_1\delta_2)^w}
\ll\frac{(\delta_1\delta_2)^{\frac12}}{(\delta\delta_1\delta_2)^{u+v-w}}\frac1{\tilde{\ell}^{u+v}}\frac{N^uM^v}{d^w}\ll\frac{N^uM^v}{d^w}
\end{equation}
in light of $u+v\geqslant w+\frac12$.

Now, \cite[Proposition 8]{BM2015} can be further refined: namely, under identical notations, we have that, for $N\geq 20M$ and still without any recourse to the Ramanujan--Petersson conjecture,
\begin{align}\label{eqpro8}
\mathcal{S}(\ell_1, \ell_2, d, N, M) \ll  (dN)^{\varepsilon}   \B(\frac{N}{d^{\frac12}} + \frac{N^{\frac54}M^{\frac14} }{d(\ell_1\ell_2)^{\frac14}} + \frac{N^{\frac34}M^{\frac14} }{d^{\frac14}} + \frac{NM^{\frac12} }{d^{\frac34}}\B).
\end{align}
Indeed, this follows by simply retaining the factor $(\ell_1\ell_2)^{\frac14}$ between \cite[(8.13)]{BM2015} and \cite[(8.17)]{BM2015}.

With the refined estimate \eqref{eqpro8} at our disposal, we now return to \eqref{4310-theta} and set $\ell_1=f_1g_1$, $\ell_2=f_2g_2$.
We again apply the decomposition as in \cite[Section 8.5]{BM2015}
\begin{equation}
\label{all-common-factors}
\ell_1 = \ell_1'\tilde{\ell} \delta \delta_1, \quad  \ell_2 = \ell_2'\tilde{\ell} \delta \delta_2, \quad d = d' \delta\delta_1\delta_2,
\end{equation}
where $\delta = (d, \ell_1, \ell_2)$, $\delta_1 = (d, \ell_1)/\delta$, $\delta_2 = (d, \ell_2)/\delta$, $\tilde{\ell} = (\ell_1, \ell_2)/\delta$.
Using the Hecke relation as in \cite[Section 8.5]{BM2015} and \eqref{toward-RPC} yields
\begin{equation}
\label{after-Hecke}
S_{N, M, d, q}\ll \frac{d }{(NM)^{\frac12}} \sum_{\substack{g_1\mid f_1 \mid q\\ g_2\mid f_2 \mid q}}   (f_1f_2\delta_1\delta_2)^{\theta + \varepsilon}\sum_{g\mid \delta_2}\sum_{h\mid \delta_1}
\Bigl|   \mathcal{S}\l(\ell'_1g, \ell'_2h, d', \frac{N}{\delta\delta_1\delta_2\tilde{\ell}}, \frac{M}{\delta\delta_1\delta_2\tilde{\ell}}\r)\Bigr|.
\end{equation}
Estimating the innermost term $\mathcal{S}$ using \eqref{eqpro8} and handling the resulting terms as in \eqref{a-bit-more-care}, we obtain
\begin{align*}
S_{N, M, d, q}
& \ll (qN)^{\varepsilon} \sum_{f_1, f_2 \mid q} \frac{(f'_1 f'_2)^{\theta } }{  (\delta\delta_1\delta_2)^{\frac12-2\theta} \tilde{\ell}^{1-2\theta}} \B( \frac{(dN)^{\frac12}}{M^{\frac12}} + \frac{N^{\frac34}  }{ (\ell_1'\ell_2')^{\frac14} M^{\frac14}} + d^{\frac34}  \frac{N^{\frac14}}{M^{\frac14} } + d^{\frac14} N^{\frac12} \B),
\end{align*}
where $f_j'=f_j/(f_j,\tilde{\ell} \delta\delta_j )$ for $j=1, 2$. It is easy to see that $f_j'\mid \ell_j'$ and $f_j'\mid q/d$, which yield $(f_1', f_2')=1$ and thus $(f_1'f_2')^{\theta} \leq \min\{(q/d)^{\theta}, (\ell_1'\ell_2')^{\theta}\}$.
Using that $\theta\leq 1/4$, dropping the denominator $(\delta\delta_1\delta_2)^{1/2-2\theta} \tilde{\ell}^{1-2\theta}$, and using, for every $u\geq\frac14$, $\sum_{d\mid q}(q/d)^{\theta}d^u\ll q^{u+\ve}$,
we conclude \eqref{eq3.11'}.

We now turn our attention to proving \eqref{eq3.12'}. The bound \cite[(3.12)]{BM2015} is a trivial estimate of \cite[(3.10)]{BM2015} using the uniform individual bound (quoted there as \cite[(3.9)]{BM2015})
\begin{equation}
\label{Blo04-bound}
\mathcal{D}(\ell_1,\ell_2,h,N,M)\ll (NMq)^{\ve}(N+M)^{\frac12+\theta},
\end{equation}
which is established in \cite[Theorem 1.3]{Blo2004} independently of the Ramanujan--Petersson conjecture. Let $\ell_1, \ell_2, h\in \mathbb{N}$ with $\delta=(\ell_1, \ell_2)$. Observing from the definition that
\begin{equation}\label{eqDdelta}
\mathcal{D}(\ell_1, \ell_2, h, N, M)=\mathcal{D}\l(\frac{\ell_1}{\delta}, \frac{\ell_2}{\delta}, \frac{h}{\delta}, \frac{N}{\delta}, \frac{M}{\delta}\r),
\end{equation}
we refine \eqref{Blo04-bound} to
\begin{equation*}
\mathcal{D}(\ell_1, \ell_2, h, N, M)\ll (NMq)^\ve\l(\frac{N+M}{\delta}\right)^{\frac12+\theta}.
\end{equation*}
Applying this to \eqref{4310-theta} yields
\[
S_{N, M, d, q}\ll\frac{d}{(NM)^{\frac12}}q^\ve\frac{N}{d}\sum_{g_1\mid f_1\mid q}\sum_{g_2\mid f_2\mid q}(f_1f_2)^\theta\l(\frac{N+M}{(f_1g_1,f_2g_2)}\right)^{\frac12+\theta}\ll \frac{q^{\theta+\ve}N^{1+\theta}}{M^{\frac12}}.
\]
This proves \eqref{eq3.12'}, and, as already commented, \eqref{eq3.11'} and \eqref{eq3.12'} combined imply the statement of the lemma.
\end{proof}

For unbalanced terms, we assume $N\ge q^{\frac32-\delta}$ for some small $\delta>0$. Applying the Voronoi summation formula to the $n$-sum transform the problem to evaluating bilinear sums as in Theorem \ref{thmmain}, where both variables have lengths comparable to $q^{\frac12}$.

Our next Lemma provides a requisite Voronoi formula, available in \cite[Lemma 2.2]{FGKM2014} or \cite[Lemma 2.3]{BFK+17a}.
\begin{lemma}[Voronoi formula]\label{lemV}
Let $q\in\mathbb{N}$, $c\in\mathbb{Z}$ with $(c,q)=1$, and let $V$ be a smooth compactly supported function.  Then
\begin{align}
\sum_{n}\lambda(n)V\l(\frac{n}{N}\r)
e\l(\frac{cn}q\r)=\frac{N}{q}\sum_{\pm}\sum_{n}\lambda(n)\mathring{V}_{\pm}\l(\frac {nN}{q^2}\r)e \B(\pm\frac{\ol{c}n}{q}\B),
\end{align}
where the Hankel-type transforms $\mathring{V}_\pm:(0,\infty)\rightarrow\mathbb{C}$ are given by
\[
\mathring{V}_\pm(y)=\int_0^{\infty}V(x)\mJ_\pm\l(4\pi\ssqrt{xy}\r)\d x
\]
with the Bessel kernels satisfying
\begin{align}\label{eqJ1}
\mJ_{+}(x)=2\pi i^{k}J_{k-1}(x), \ \ \ \ \mJ_{-}(x)=0
\end{align}
for $f$ a cuspidal holomorphic newform of weight $k$, and
\begin{align}\label{eqJ2}
\mJ_+(x)=\frac{\pi i}{\sinh(\pi t)}(J_{2i\kappa}(x)-J_{-2i\kappa}(x)),\quad \mJ_-(x)=4\cosh(\pi \kappa)K_{2i\kappa}(x)
\end{align}
for $f$ a cuspidal Maa{\ss} newform with spectral parameter $\kappa$.
\end{lemma}
The Hankel-type transforms $\mathring{V}_\pm(y)$  are well-known to be of Schwartz class.
Since the Selberg conjecture of exceptional eigenvalues is known for $\rm{SL}_2(\mathbb{Z})$, the bound (see \cite[Lemma 2.4]{BFK+17a})
\[
y^j\mathring{V}_\pm(y)\ll_{i,j}(1+y)^{\frac j2}\l(1+y^{\frac12}q^{-\ve}\r)^{-i}
\]
holds for all $i,j\ge0$ and $y>0$, when $f$ is a cuspidal newform (holomorphic or Maa{\ss}) of level one and $V^{(j)}(y)\ll_jq^{j\ve}$. In particular, we have that $\mathring{V}_{\pm}(y)\ll_j y^{-j}$ for $y\ge q^{\ve}$.

\subsection{ Proof of Theorem~\ref{thmL}}
\label{thmLproof-subsec}
In this section, we combine all the preparatory steps from \S\ref{reduction-subsec} and the input from Theorem~\ref{thmmain} to prove Theorem~\ref{thmL}.

\begin{proof}[Proof of Theorem \ref{thmL}]
Recall that $\eta=\frac1{216}$. As explained in \S\ref{reduction-subsec} (see the reduction to \eqref{mBpm}), the proof of  Theorem \ref{thmL} reduces to establishing the bound
\begin{align}\label{eqeta}
\mB^{\pm}(M,N)\ll q^{-\eta+\ve}.
\end{align}
We parametrize the exponents as
\[
M=q^u, \quad N=q^v,\quad v^*=2-v.
\]
From \eqref{eqtbSMN} and \eqref{eqbalancedS}, the desired bound \eqref{eqeta} follows immediately if either $u+v\le 2-2\eta-2\theta v$ or $v-u\le 1-4\eta$. Thus, we may restrict our consideration to the complementary regime
\begin{align*}
2-2\eta-2\theta v \le u+v\le 2+\ve, \quad 1-4\eta\le v-u,
\end{align*}
which yields the key constraints
\begin{align}\label{eqvu}
u\le 1/2+2\eta+\ve, \quad u-\ve\le v^*\le 1/2+\frac{3\theta+6\eta}{2+2\theta}, \quad u+v^*\le 1+4\eta.
\end{align}

After rearrangement in \eqref{mBpm}, we derive
\begin{align*}
\mB^{\pm}(M,N)=\frac{1}{\vp^{*}(q)}\sum_{d\mid q}\mu\l(\frac qd\r)\frac{\vp(d)}{(MN)^{\frac12}}\sum_{(m,q)=1}\lambda_1(m)W_1\l(\frac mM\r)\sum_{\substack{n\equiv \pm m\ppmod d\\ (n,q)=1}}\lambda_2(n)W_2\l(\frac nN\r).
\end{align*}
Given $(m,q)=1$ and $n\equiv\pm m\ (\bmod\ d)$, the condition $(n,q)=1$ is equivalent to $(n,q_d)=1$. Applying M\"{o}bius inversion and the Hecke relation yields
\begin{align*}
\sum_{\substack{n\equiv \pm m\ppmod d\\ (n,q)=1}}\lambda_2(n)W_2\l(\frac nN\r)
&=\sum_{f\mid q_d}\mu(f)\sum_{\substack{fn\equiv \pm m\ppmod d}}\lambda_2(fn)W_2\l(\frac {fn}N\r)\\
&=\sum_{g\mid f\mid q_d}\mu(f)\mu(g)\lambda_2\l(\frac fg\r)\sum_{\substack{n\equiv \pm \ol{fg}m\ppmod d}}\lambda_2(n)W_2\l(\frac {fgn}N\r).
\end{align*}
We detect the condition $n\equiv\pm \ol{fg}m\ (\bmod\ d)$ with primitive additive characters modulo $r$ for $r\mid d$, and then apply the Voronoi summation formula (Lemma~\ref{lemV}) to the $n$-sum, reducing it to the sum of four sums of the form
\[
\frac1 d\sum_{r\mid d}\frac{N}{fgr^{\frac12}}\sum_{n\ge1}\lambda_2(n)\mathring{W}_{2,\pm}\l(\frac{n}{N^*}\r)\K(\pm \ol{fg}mn;r),
\]
where $N^*=fgr^2/N\ll q^{v^*}$ since $fr\le q$, and the two $\pm$ signs vary independently of each other. Substituting this back into $\mB^{\pm}(M,N)$ gives
\begin{align*}
\mB^{\pm}(M,N)\ll\sum_{r\mid q}\sum_{f\mid g\mid q_r}\frac{q^\ve f^{\theta}r^{\frac12}}{q(fgMN^*)^{\frac12}}\B|\sum_{m\le M}\sum_{n\le N^*}\alpha_m\beta_n\K(\pm \ol{fg}mn;r)\B|
\end{align*}
with coefficients
\[
\alpha_m=\lambda_1(m)W_1\l(\frac mM\r),\quad \beta_n=\lambda_2(n)\mathring{W}_{2,\pm}\l(\frac{n}{N^*}\r).
\]
Trivial summation via Weil's bound for Kloosterman sums produces
\[
\mB^{\pm}(M,N)\ll (M/N)^{\frac12}f^{\theta}r^{\frac32}q^{-1+\ve}\le r^{\frac32-\theta}q^{-\frac32+\theta+2\eta+\ve},
\]
establishing \eqref{eqeta} unless
\begin{align}\label{eqconditionr}
r\ge q^{1-\frac{6}{3-2\theta}\eta}.
\end{align}
Given $\eta$'s small magnitude, \eqref{eqconditionr} combined with \eqref{eqvu} ensures the condition \eqref{conditionMN} in Theorem \ref{thmmain}. Thus, we derive from Theorem \ref{thmmain} and \eqref{eqvu} that
\begin{align*}
\mB^{\pm}(M,N)&\ll\sum_{r\mid q}\frac{r^{\frac12}(MN^*)^{\frac12}}{q}\l(M^{-\frac12}r^{\frac1{6}}+M^{-\frac3{25}}(N^*)^{-\frac3{10}}r^{\frac15}+(MN^*)^{-\frac3{16}}r^{\frac{11}{64}}\r)\\
&\ll q^{-\frac13+\frac12 v^*+\ve}+q^{-\frac3{10}+\frac{9}{50}u+\frac15(u+v^*)+\ve}+q^{-\frac{21}{64}+\frac5{16}(u+v^*)+\ve}\le q^{-\eta+\ve}
\end{align*}
for $v^*\le \frac23-2\eta$ with $\eta= \frac1{216}$. The remaining case $v^*>\frac23-2\eta$ can be effectively handled using the estimate \eqref{eqPoly}, which is deduced from the P\'olya--Vinogradov method.
\end{proof}

\begin{remark}
\label{PQ-remark}
It remains to justify the claims
\[ P(1),Q(1)=\exp(O_{\ve}((\log\log q)^{\ve})),\quad  P'(1)/P(1)=O_{\ve}((\log\log q)^{1+\ve}) \]
for the finite Euler products defined in \eqref{euler-products-eq}. Recall that
\[
\lambda(p^2)=\lambda(p)^2+1,\quad \lambda(p)\ll p^{\theta}
\]
for an absolute $\theta<1/4$ (where $\theta=7/64$ is admissible). Choosing any $A>1/(1-2\theta)$, we have that $\lambda(p^2)/p\ll 1/\log q$ for all $p\ge\log^Aq$, whence, keeping in mind the simple estimate $\omega(q)\ll\log q$,
\begin{align*}
\sum_{p\mid q}\frac{\lambda(p^2)}{p}&\ll \sum_{p\le \log^Aq}\frac{\lambda(p)^2+1}{p}.
\end{align*}
Taking a large integer $r$ and rewriting every product $p_1\dots p_r=nm$ with $n$ square-free, $m$ square-full, and $(n,m)=1$, we have that
\begin{align*}
\Bigg(\sum_{p\le \log^Aq}\frac{\lambda(p)^2+1}{p}\Bigg)^r
&=\mathop{\sum\cdots\sum}_{p_1,\cdots,p_r\le \log^Aq}\frac{(\lambda(p_1)^2+1)\cdots(\lambda(p_r)^2+1)}{p_1\cdots p_r}\\
&\ll_r\sum_{n\le \log^{Ar}q}\frac{1*\lambda^2(n)}{n}\sum_{m ~\text{square-full}}\frac1{m^{1-2\theta}}\\
&\ll_r\log\log q\sum_{n\le \log^{Ar}q}\frac{\lambda(n)^2}{n}\ll_r(\log\log q)^2.
\end{align*}
Taking $r>2/\ve$ yields\footnote{We note that, in fact, $\theta<1/2$ suffices to reach the conclusions that follow, by instead separating out $s$-power-full numbers with a fixed $s>1/(1-2\theta)$ and then taking $r>2s/\ve$.}
\begin{align}\label{eqlp}
\sum_{p\le \log^Aq}\frac{\lambda(p)^2+1}{p}\ll_\ve (\log\log q)^{\ve},
\end{align}
hence
\begin{gather*}
P(1), Q(1)=\exp\Bigg(O\B(\sum_{p\mid q}\frac{\lambda(p^2)}p\B)+O(1)\Bigg)=\exp\big(O_{\ve}\big((\log\log q)^{\ve}\big)\big),\\
\frac{P'(1)}{P(1)}\ll \sum_{p\mid q}\frac{\lambda(p^2)\log p}{p}+O(1)\ll_{\ve}(\log\log q)^{1+\ve}.
\end{gather*}
\end{remark}

\end{document}